\documentclass[11pt,a4paper]{amsart}[2010]
\usepackage{tikz}
\usepackage{enumitem}
\usepackage{hyperref}

\title{Rokhlin-type properties, approximate innerness and $\Zh$-stability}

\author{Ilan Hirshberg}

\address{Department of Mathematics, Ben Gurion University of the Negev, \phantom{----------------}\linebreak\text{}\hspace{3.5mm}
P.O.B. 653, Be'er Sheva 84105, Israel}
\email{ilan@math.bgu.ac.il}

\thanks{
	This research was supported by the Israel Science Foundation grant no.~476/16, by the 
	 Mittag-Leffler Institute (as part of the 2016 program on classification of operator algebras), and by the Fields Institute}

\usepackage{amsmath}
\usepackage{amssymb}
\usepackage{amsthm}
\usepackage[all]{xypic}

\allowdisplaybreaks

\theoremstyle{plain}

\newtheorem{Thm}{Theorem}[section]
\newtheorem{Cor}[Thm]{Corollary}
\newtheorem{Lemma}[Thm]{Lemma}

\newtheorem{Prop}[Thm]{Proposition}

\theoremstyle{definition}
\newtheorem{Def}[Thm]{Definition}

\newtheorem{Exl}[Thm]{Example}
\newtheorem{Rmk}[Thm]{Remark}
\newtheorem{Question}[Thm]{Question}

\newcommand{\Zh}{\mathcal{Z}}

\newcommand{\T}{{\mathbb T}}
\newcommand{\R}{{\mathbb R}}
\newcommand{\N}{{\mathbb N}}
\newcommand{\Z}{{\mathbb Z}}
\newcommand{\C}{{\mathbb C}}

\newcommand{\aut}{\mathrm{Aut}}

\newcommand{\eps}{\varepsilon}
\numberwithin{equation}{section}

\newcommand{\dimrokc}{\dim^{\mathrm{c}}_{\mathrm{Rok}}}

\newcommand{\Ann}{\mathrm{Ann}}

\newcommand{\Div}{\mathrm{Div}}
\newcommand{\Cu}{\mathrm{Cu}}

\begin{document}
\begin{abstract}
We establish four results concerning connections between actions on separable $C^*$-algebras  with  Rokhlin-type properties and absorption of the Jiang-Su algebra $\Zh$. For actions of residually finite groups or of the reals which have finite Rokhlin dimension with commuting towers, we show that if the action of any nontrivial group element is approximately inner then the $C^*$-algebra acted upon is $\Zh$-stable. Without the assumption on approximate innerness, we show that the crossed product has good divisibility properties under mild assumptions. We also establish an analogous result for the generalized tracial Rokhlin property and tracial versions of approximate innerness and $\Zh$-absorption for actions of finite groups and of the integers.
For actions of a single automorphism which have the Rokhlin property, we show that a condition which is strictly weaker than requiring that some power of the automorphism is approximately inner is sufficient to obtain that the crossed product absorbs $\Zh$ even when the original algebra is not.
\end{abstract}

\maketitle

The property of $\Zh$-stability or $\Zh$-absorption, that is, absorbing tensorially the Jiang-Su algebra (\cite{Jiang-Su}), is a key regularity property for $C^*$-algebra, and is viewed as the analogue of being McDuff for von Neumann algebras. For simple separable and nuclear $C^*$-algebras,  $\Zh$-stability is now known to be equivalent to having finite nuclear dimension (\cite{Winter-nuc-dim-implies-Z-stable,CETWW}). Capping decades of research, for $C^*$-algebras which furthermore satisfy the Universal Coefficient Theorem, it has been shown that those properties are sufficient to establish classfiability in terms of the Elliottt invariant (\cite{TWW,EGLN:arXiv}). Though the main interest in in $\Zh$-stability has been in the context of the Elliott classification program, $\Zh$-stability is an important regularity property in the context of non-nuclear $C^*$-algebras as well. For instance, nuclearity is not needed in order to establish the implications of $\Zh$-stability from \cite{Rordam-stable-real-rank}.

The Rokhlin property (see for instance \cite{kishimoto-rokhlin}), later generalized to Rokhlin dimension (\cite{HWZ,SWZ}), can be viewed as a regularity property for group actions on $C^*$-algebras. In the context of establishing $\Zh$-stability for crossed products, Rokhlin dimension and other variants of the Rokhlin property have come to play in two principal ways. One way is to establish permanence results: if $A$ is separable and $\Zh$-stable and $\alpha$ is an action of a group $G$ on $A$ which has finite Rokhlin dimension with commuting towers, then the crossed product is $\Zh$-stable as well. This holds for actions of countable residually finite groups which have a box space of finite asymptotic dimension and for actions of $\R$; we refer the reader to \cite{HWZ,hirshberg-phillips,SWZ,HSWW} for precise statements and definitions. In the context of the Rokhlin property, this was considered earlier in \cite{HW-PJM}, and we refer the reader to \cite{Sato,Matui-Sato-Z-stability} for related work which uses a tracial version of the Rokhlin property.
 A different way in which Rokhlin dimension comes about in this context is to establish permanence of finite nuclear dimension: if $A$ has finite nuclear dimension then so does the crossed product. If the crossed product is furthermore simple, one uses \cite{Winter-nuc-dim-implies-Z-stable} (or \cite{tikuisis-non-unital} in the non-unital case) to deduce that the crossed product is in fact $\Zh$-stable. This is one approach for showing that crossed products of $C(X)$ by minimal homeomorphisms are $\Zh$-stable, where $X$ is compact metric with finite covering dimension; see \cite{HWZ,Szabo,SWZ,HSWW} (and \cite{Toms-Winter-GAFA,Elliott-Niu-mean-dimension,kerr-szabo} for papers which obtain related results without Rokhlin dimension techniques).
 
 In this paper, we further investigate the connection between Rokhlin-type properties and $\Zh$-stability. In Section \ref{section Rokhlin dimension and approximate innerness}, we show that if $A$ admits an action of a residually finite group with finite Rokhlin dimension with commuting towers, and the action of at least one nontrivial group element is approximately inner, then in fact $A$ has to be $\Zh$-stable (and in particular, the crossed product is $\Zh$-stable).  

In Section \ref{section:tracial results}, we provide tracial analogues for the results in Section \ref{section Rokhlin dimension and approximate innerness}. Namely, if $A$ is simple unital with stable rank 1 and admits an action of $\Z$ or of a finite group with the generalized tracial Rokhlin property such that the action of at least one nontrivial group element is strongly tracially approximately inner (in the sense of \cite{phillips-tracial}) then $A$ has to be tracially $\Zh$-absorbing (which, if $A$ is also nuclear, means that $A$ is $\Zh$-stable).

The results of Sections \ref{section Rokhlin dimension and approximate innerness} and \ref{section:tracial results} should best be viewed as obstructions, rather as a way of establishing that a given $C^*$-algebra is $\Zh$-stable. The Rokhlin property and its various variants play an important role in studying group actions and their associated crossed products. One reason for that is that those properties are quite common. For example, the  automorphisms with Rokhlin dimension at most $1$ with commuting towers of a unital separable $\Zh$-stable form a dense $G_{\delta}$ set in the automorphism group (\cite{HWZ}). The results in this paper show that the situation in the non-$\Zh$-stable case is very different. This is not to say that they do not exist. For example, the recent paper \cite{AGP2019} establishes results concerning the behavior of the radius of comparison when forming crossed products by actions of finite groups with the weak tracial Rokhlin property; the paper contains examples of such actions which have the Rohklin property, however the construction is quite involved.  It is not known whether there exists a simple separable unital nuclear $C^*$-algebras with positive radius of comparison which admits a finite group action with the weak tracial Rokhlin property but such the action on the trace space is trivial (\cite[Question 7.2]{AGP2019}). While we do not fully answer the question here, the results herein show that the most natural first approach to construct such an example, namely an approximately inner action (or even one which is only strongly tracially approximately inner) cannot work. 
 
 Without assuming any kind of approximate innerness, one cannot possibly expect $\Zh$-stability: indeed, any free action on a compact metric space with finite covering dimension gives rise to an action on $C(X)$ which has finite Rokhlin dimension (with commuting towers, which we get for free in this case), and $C(X)$ is obviously not $\Zh$-stable. Furthermore, the crossed product need not be $\Zh$-stable either. 
 We nevertheless show in Section \ref{section Rokhlin dimension and divisibility conditions} that finite Rokhlin dimension with commuting towers is sufficient to deduce that the crossed product has good divisibility properties in the sense studied by Robert and R{\o}rdam in \cite{Robert-Rordam}.
 In Section \ref{section Nearly approximately inner automorphisms}, we introduce a weaker property than approximate innerness, which, together with the Rokhlin property, is sufficient for showing that the crossed product is $\Zh$-stable. This property holds for examples which can be viewed as noncommutative odometers. We provide, for instance, examples of such actions on reduced group $C^*$-algebras for a family of ICC groups introduced by Dixmier and Lance, whose associated group factors have property $\Gamma$ but are not McDuff. As the Rokhlin property is a statement about the induced action on the central sequence algebra, this provides a highly non-commutative example, quite far from the $\Zh$-stable setting, in which we can nonetheless establish $\Zh$-stability results for the associated crossed product.
 
 We now recall some definitions and fix notations. 
To lighten notation, we sometimes write $x \approx_{\eps} y$ to mean $\|x-y\|<\eps$.  By an \emph{order zero} map we always mean a completely positive contractive order zero map; we refer the reader to \cite{winter-zacharias-order-zero} for more on order zero maps.
 
 For two positive integers $p,q$ we denote the associated dimension drop algebra by
 $$
 I_{p,q} = \{f \colon [0,1] \to M_p \otimes M_q \mid f(0) \in M_p \otimes 1 , f(1) \in 1 \otimes M_q\}
 $$
Those dimension drop algebras form the building blocks in the construction of the Jiang-Su algebra in \cite{Jiang-Su}.

We use the following definitions for the corrected central sequence algebra from \cite{Kirchberg-abel}, and the modified version introduced in \cite[Definition 5.1]{SWZ}, though we use slightly different notation. We note that we use an ultrafilter, whereas in \cite{SWZ} one uses the cofinite filter, but this makes no difference for any of the proofs herein. We fix a free ultrafilter $\omega$. Let $A$ be a separable $C^*$-algebra. Let $D \subset A_{\omega}$ be a separable $C^*$-subalgebra. Set $\Ann(D) \subset A_{\omega}$ to be the annihilator of $D$, that is, $\Ann(D) = \{x \in A_{\omega} \mid \forall y \in D: \; xy = yx =0 \}$. Then $\Ann(D)$ is an ideal in $A_{\omega} \cap D'$. We denote $F(A,D) = A_{\omega} \cap D' / \Ann(D)$. This is a unital $C^*$-algebra. If $D = A$, the canonical copy of $A$ in $A_{\omega}$, we write $F(A) = F(A,A)$ for short. (If $A$ is unital, then we simply have $F(A) = A_{\omega} \cap A'$.) If $D \subseteq E \subset A_{\omega}$ are separable $C^*$-algebras, notice that $\Ann(D) \supseteq \Ann(E)$ and $A_{\omega} \cap D' \supseteq A_{\omega} \cap E'$, and therefore there exists a canonical unital homomorphism $F(E,A) \to F(D,A)$.

If $G$ is a group and $\alpha \colon G \to \aut(A)$ is an action, then $\alpha$ induces actions on $A_{\omega}$, on $F(A)$ and on $F(D,A)$ for any separable $\alpha$-invariant $C^*$-subalgebra of $D$ of $A_{\omega}$. By slight abuse of notation, we denote those actions by $\alpha$ as well.

Let $A$ be a separable $C^*$-algebra, and let $\alpha \in \aut(A)$ be an automorphism (thought of as an action of $\Z$). The automorphism $\alpha$ is said to have the \emph{Rokhlin property} if for any $m \in \N$ there exists a partition of unity of projections $e_{0,0},e_{0,1},\ldots,e_{0,m-1},e_{1,0},e_{1,1},\ldots,e_{1,m} \in F(A)$ such that $\alpha(e_{j,k}) = e_{j,k+1}$ for $j=0,1$ and $k=0,1,\ldots,m-1+j$. Notice that this implies that $\alpha(e_{0,m-1}+e_{1,m} ) = e_{0,0} + e_{1,0}$. We refer to such a collection of projections as a double Rokhlin tower.
The single tower version of the Rokhlin property is what we get when we set $e_{1,k} = 0$ for $k=0,1,\ldots,m$, and it is much more restrictive than the Rokhlin property. We refer the reader to \cite{HWZ} for discussions of various flavors of Rokhlin dimension. There is a definition of Rokhlin dimension for actions of $\Z$ which generalizes the Rokhlin property, in that Rokhlin dimension zero is precisely the Rokhlin property, and a single tower version which generalizes the single tower Rokhlin property, but such that the Rokhlin property without single towers only implies Rokhlin dimension 1. Deviating from the notation in \cite{HWZ}, but following the convention in \cite{SWZ}, our notion of Rokhlin dimension refers to the single tower version (however, when we discuss the Rokhlin property, we assume the general one which may have two towers). The definition is discussed in section \ref{section Rokhlin dimension and approximate innerness}. When considering higher Rokhlin dimension, the difference between the single tower and double tower versions is less significant, because if one is finite then so is the other, and for most uses of Rokhlin dimension we only care about whether it is finite or not. The difference between the commuting and noncommuting tower versions is more significant (\cite{hirshberg-phillips}), and here we use the stronger, commuting tower version. 

I thank G\'{a}bor Szab\'{o} for some helpful remarks.

\section{Rokhlin dimension and approximate innerness}
\label{section Rokhlin dimension and approximate innerness}

Our goal in this section is to prove Theorem \ref{Thm:Z-absorption}. The main tool we use is that of universal spaces for actions with finite Rokhlin dimension with commuting towers.

Universal spaces for actions of finite groups which have finite Rokhlin dimension with commuting towers were introduced in \cite{hirshberg-phillips}, and studied further in \cite{GHS}. We refer the reader to the discussion after Definition 3.2 in \cite{GHS}.  For any finite group $G$ and non-negative integer $d$ there exists a free $G$-space $X_{G,d}$, which is a finite $d$-dimensional simplicial complex, such that an action of $G$ on a separable $C^*$-algebra $A$ has Rokhlin dimension with commuting towers at most $d$ if and only if there exists a unital equivariant homomorphism $C(X_{G,d}) \to F(A)$. 

One can introduce similar spaces for actions of more general residually finite groups. If $G$ is a group and $H$ is a finite index subgroup, then we recall from \cite[Definition 5.4]{SWZ} that an action $\alpha$ of $G$ on a separable $C^*$-algebra $A$ is said to have finite Rokhlin dimension at most $d$ with commuting towers relative to $H$ (denoted $\dimrokc(\alpha,H) \leq d$) if there exist equivariant completely positive contractive order zero maps
\[
\varphi^{(l)} \colon (C(G/H),G\mathrm{-shift}) \to (F(A),\alpha) \quad \mathrm{ for } \; l=0,1,\ldots,d
\]
with commuting images such that $\sum_{l=0}^d \varphi^{(l)}(1) = 1$. This is equivalent to requiring that there exists commuting positive contractions $\{f_s^{(l)}\}_{s \in G/H \; , \; l=0,1,\ldots,d} \subseteq F(A)$ such that 
\begin{enumerate}
	\item $f_s^{(l)} f_{s'}^{(l)} = 0$ if $s \neq s'$.
	\item $\sum_{l=0}^d \sum_{s \in G/H} f_s^{(l)} = 1$.
	\item $\alpha_g (f_s^{(l)}) = f_{g \cdot s}^{(l)}$ for all $g \in G$ and for all $s \in G/H$.
\end{enumerate}
The action $\alpha$ is then said to have Rokhlin dimension with commuting towers at most $d$, denoted $\dimrokc(\alpha) \leq d$, if $\dimrokc(\alpha,H) \leq d$ for all subgroups $H$ of finite index.

Consider the universal $C^*$-algebra generated by elements $\{f_s^{(l)}\}_{s \in G/H \; , \; l=0,1,\ldots,d}$ satisfying the above conditions. This universal $C^*$-algebra carries an action of $G$. We define $X = X_{G,H,d}$ to be the Gelfand spectrum of this $C^*$-algebra. The space $X$ has the structure of a $d$-dimensional simplicial complex, which can be described as follows. The vertices are given by the elements $\{ (s,l) \}_{s \in G/H \; , \; l=0,1,\ldots,d}$. A $k$-tuple of distinct vertices $(s_1,l_1),(s_2,l_2),\ldots,(s_k,l_k)$ defines a simplex if and only if the numbers $l_1,l_2,\ldots,l_k$ are distinct (which happens if and only if $\prod_{j=1}^{k} f_s^{(l)} \neq 0$). The action of $G$ on $X$, which is determined by its action on the vertices, is given by $g \cdot (s,l) = (g\cdot s,l)$. We denote by $\sigma$ the action of $G$ this defines on $C(X)$.

Thus, we have $\dimrokc(\alpha,H) \leq d$ if and only if there exists an equivariant unital homomorphism from $C(X_{G,H,d})$ to $F(A)$. The same argument as in \cite[Remark 5.6]{SWZ}, shows that for any separable $\alpha$-invariant $C^*$-subalgebra $D$ of $F(A)$, we can arrange for the image of the homomorphism to lie in $F(D,A)$.

\begin{Lemma}
	\label{Lemma:restrict-to-subgroup}
	Let $A$ be a separable $C^*$-algebra. Let $G$ be a residually finite group, and let $\alpha \colon G \to \aut(A)$ be an action which has finite Rokhlin dimension with commuting towers. Suppose $g_0$ is a nontrivial group element. Then the restriction of $\alpha$ to the subgroup generated by $g_0$ also has finite Rokhlin dimension with commuting towers.
\end{Lemma}
\begin{proof}
	We first observe that if $H<G$ is a subgroup of finite index, then there exists an $\left < g_0 \right >$-equivariant embedding $C(\left < g_0 \right > / \left < g_0 \right >\cap H) \to C(G/H)$. If $g_0$ has finite order then choose $H < G$ such that $\left < g_0 \right > \cap H = \{1\}$. Then the group generated by $g_0$ acts freely on $G/H$ and therefore on $X_{G,H,d}$ for $d = \dimrokc(\alpha,H)$. Therefore, by \cite[Lemma 1.9]{hirshberg-phillips}, we have $\dimrokc(\alpha_{g_0}) \leq d$. 
	
	We now assume that $g_0$ has infinite order.
	The set of subgroups of $\left < g_0 \right >$ of the form $H \cap \left < g_0 \right >$ for $H<G$ of finite index form a regular approximation in the sense of \cite[Notation 3.6]{SWZ}. Since the asymptotic dimension of the box space of $\left < g_0 \right >$ is $1$, it follows from the proof of \cite[Theorem 7.2]{SWZ} that the restriction of $\alpha$ to $\left < g_0 \right >$ has finite Rokhlin dimension with commuting towers. In fact, $\dimrokc(\alpha_{g_0}) \leq 2\dimrokc(\alpha)+1$. (The statement of \cite[Theorem 7.2]{SWZ} concerns Rokhlin dimension without necessarily commuting towers, but an inspection of the proof shows that the same result holds for the commuting tower version as well.) 
\end{proof}

The following is a straightforward consequence of \cite[Theorem 1]{Dadarlat-Toms}.
\begin{Lemma}
	\label{Lemma:DT09}
	Let $A$ be a separable $C^*$-algebra. Suppose $F(A)$ contains a unital copy of a subhomogeneous algebra without characters. Then $A \cong A \otimes \mathcal{Z}$.
\end{Lemma}
\begin{proof}
	Let $B \subset F(A)$ be a unital subhomogeneous subalgebra without characters. By \cite[Corollary 1.13]{Kirchberg-abel}, $F(A)$ contains a unital copy of the infinite tensor power $B^{\otimes \infty}$. By \cite[Theorem 1]{Dadarlat-Toms}, $B^{\otimes \infty}$ is $\Zh$-absorbing, and in particular, contains a unital copy of $\mathcal{Z}$. By \cite[Proposition 4.4]{Kirchberg-abel} (see also \cite[Proposition 4.1]{HRW}), it follows that $A$ is $\Zh$-absorbing, as required.
\end{proof}

The following fact is presumably well-known, although it seems difficult to track down in the literature. The question of when a crossed product of $C(X)$ by a non-free action of a group is a continuous trace algebra was considered in \cite{Willimas-1981}, however here we need a more precise formulation in our specific case. 
\begin{Prop}
	\label{Prop:subhomogeneous}
	Let $X$ be a locally compact Hausdorff space. Let $h$ be a free action of $\Z_n$ on $X$. Let $\alpha$ be the action of $\Z$ on $C(X)$ given by $\alpha(f) = f \circ h^{-1}$ (noting that $\alpha$ is $n$-periodic). Then $C(X) \rtimes_{\alpha} \Z$ is a locally trivial $n$-homogeneous algebra with spectrum $X/\Z_n \times \T$.
\end{Prop}
For a proof, we refer the reader to the proof of \cite[Proposition 2.1]{Hirshberg-Wu}. The statement of \cite[Proposition 2.1]{Hirshberg-Wu} says that the crossed product in question is subhomogeneous, not homogeneous, since the assumptions are weaker. However, the proof shows that the possible dimensions of irreducible representations are precisely the lengths of the orbits which occur. Therefore, if all orbits are of the same size, the resulting crossed product is in fact homogeneous.

The following lemma is a straightforward consequence of a simple reindexation trick (or the Kirchberg $\eps$-test), and we omit the proof.
\begin{Lemma}
	\label{Lemma:reindexation}
  Let $A$ be a separable $C^*$-algebra, and let $\alpha$ be an automorphism of $A$. 
  \begin{enumerate}
 	\item If $D$ is a unital $C^*$-algebra and there exists a unital homomorphism $D \to F(A)$ then for any separable subspace $S \subset F(A)$ there exists a unital homomorphism $D \to F(A) \cap S'$.
 	\item If $\alpha$ is approximately inner then for any separable subspace $S \subset A_{\omega}$ there exists a unitary $u \in A_{\omega}^+$ such that $uxu^* = \alpha(x)$ for all $x \in S$. 
  \end{enumerate}
\end{Lemma}

We discussed the case of discrete group actions above. For the case of flows, that is, actions of $\R$, we refer the reader to \cite[Definition 6.1]{HSWW}.

\begin{Thm}
	\label{Thm:Z-absorption}
	Let $A$ be a separable $C^*$-algebra. Let $G$ be a residually finite group or $G = \R$, and let $\alpha \colon G \to \aut(A)$ be an action which has finite Rokhlin dimension with commuting towers. Suppose there exists $ g_0 \neq 1$ such that $\alpha_{g_0}$ is approximately inner. Then $A \cong A \otimes \Zh$.
\end{Thm}
\begin{proof}
	We may assume that $G = \left < g_0 \right >$, since the restriction to $\left < g_0 \right >$ has finite Rokhlin dimension with commuting towers:	if $G$ is discrete, this is Lemma \ref{Lemma:restrict-to-subgroup}, and if $G = \R$, then this follows by a straightforward modification of \cite[Proposition 5.5]{HSWW} to the commuting tower case.
	We denote $\alpha = \alpha_{g_0}$ and set $d = \dimrokc(\alpha)$.
	
	If $g_0$ has finite order, set $n = |G|$, otherwise fix any natural number $n>1$.  Notice that we can identify $X_{\Z_n,d}$ with $X_{\Z,n\Z,d}$ (the former has a free action of $\Z_n$, which we can extend to a periodic action of $\Z$).
	
	Since $\alpha$ is approximately inner, there exists a unitary $u \in A_{\omega}^+$ such that $uxu^* = \alpha(x)$ for all $x \in A$. Fix such a unitary $u$.

	Let $D \subset A_{\omega}$ be the $C^*$-algebra generated by the elements of the form $\alpha^k(au)$ for all $a \in A$ and for all $k \in \Z$. Choose an equivariant unital homomorphism  $\varphi \colon C(X_{\Z,n\Z,d}) \to F(D,A)$.   We use the Choi-Effros theorem to choose a completely positive lifting $\Phi \colon C(X_{\Z,n\Z,d})  \to A_{\omega} \cap D'$. Let $E$ be the smallest $\alpha$-invariant $C^*$-algebra containing $D$ and the image of $\Phi$. Now, using again the fact that $\alpha$ is approximately inner, we can choose a unitary $v \in  A_{\omega}^+$ such that $vxv^* =\alpha(x)$ for all $x \in E$. 
	
	Notice that $A \subseteq D$, and $u^*v \in A_{\omega}^+ \cap A'$.  Denote by $w$ the image of $u^*v$ in $F(A)$ under the quotient map by $\Ann(A)$. Let $\psi \colon C(X_{\Z,n\Z,d}) \to F(A)$ be the composition of $\varphi$ with the canonical unital homomorphism $F(D,A) \to F(A)$. By construction, we have $w x w^* = \alpha(x)$ for all $x \in \psi( C(X_{\Z,n\Z,d}) )$. As $(\psi,w)$ is a covariant representation of $(C(X_{\Z,n\Z,d}),\sigma)$, it follows that $C^*( C(X_{\Z,n\Z,d}) , w)$ is a quotient of the crossed product $ C(X_{\Z,n\Z,d}) \rtimes_{\sigma} \Z$. Since, by Proposition \ref{Prop:subhomogeneous}, the crossed product  $ C(X_{\Z,n\Z,d}) \rtimes_{\sigma} \Z$ is an $n$-homogeneous $C^*$-algebra, any non-zero quotient is $n$-homogeneous as well. Therefore, by Lemma \ref{Lemma:DT09}, it follows that $A \cong A \otimes \Zh$, as required.
\end{proof}

\begin{Rmk}
	Under the conditions of Theorem \ref{Thm:Z-absorption}, it follows from \cite[Theorem 10.6]{SWZ} under the further assumption that $G$ is countable and has a box space with finite asymptotic dimension and from \cite[Theorem 6.3]{HSWW} for the case $G = \R$ that the crossed product $A \rtimes_{\alpha} G$ absorbs $\Zh$ tensorially as well.
\end{Rmk}
\begin{Question}
	Is the commuting tower condition required in Theorem \ref{Thm:Z-absorption}?  The proof of Theorem \ref{Thm:Z-absorption} suggests the following more specific question. Given $n,d \in \N$, let $A_{n,d}$ be the universal unital $C^*$-algebra given by positive contractions $\{ f_k^{(l)} \mid k=0,1,\ldots,n-1 ; l = 0,1,\ldots,d\}$, subject to the relations:
	\begin{enumerate}
		\item $ f_k^{(l)} f_j^{(l)} = 0$ for all $l = 0,1,\ldots,d$ whenever $k \neq j$.
		\item $\sum_{l=0}^d  \sum_{k=0}^{n-1} f_k^{(l)}  = 1$.
	\end{enumerate}
	We have an $n$-periodic automorphism $\sigma$ of $A_{n,d}$ given by $\sigma( f_k^{(l)} ) = f_{k+1}^{(l)}$ for all applicable indices $k,l$, where we think of addition as taken modulo $n$. Given $d$, is there an $n$ such that the crossed product $A_{n,d} \rtimes \Z$ admits a unital homomorphism from a prime dimension drop algebra?
\end{Question}

\section{The generalized tracial Rokhlin property}
\label{section:tracial results}

In this section, we provide analogous results for those in Section \ref{section Rokhlin dimension and approximate innerness} for actions which have the generalized (or weak) tracial Rokhlin property (\cite[Definitions 5.2 and 6.1]{hirshberg-orovitz}). We restrict our attention to actions of finite groups and of $\Z$, so as to avoid fine distinctions between generalizations of the tracial Rokhlin property which appeared in the literature for more general amenable groups (see for example \cite{gardella-hirshberg}). Our goal in this section is to establish an analogue of Theorem \ref{Thm:Z-absorption} for tracial versions of the Rokhlin property, approximate innerness and $\Zh$-stability. We first recall the definition of tracial $\Zh$-absorption (\cite[Definition 2.1]{hirshberg-orovitz}) and establish some simple lemmas. As usual, we use the notation $\precsim$ to denote Cuntz subequivalence.

\begin{Def}
	\label{def:tZA}
	A unital $C^*$-algebra $A$ is said to be \emph{tracially $\Zh$-absorbing} if $A \neq \C$ and for any $n \in \N$, for any nonzero positive element $a \in A$, for any $\eps>0$ and for any finite set $F \subset A$ there exists an order zero map $\varphi \colon M_n \to A$ such that:
	\begin{enumerate}
	\item $1 - \varphi(1) \precsim a$ ;
	\item For any $x \in M_n$ and for any $y \in F$ we have $\| [\varphi(x),y] \| < \eps \|x\| $ .
	\end{enumerate}
\end{Def}
Our first goal is to show that in Definition \ref{def:tZA}, it suffices to check that the definition holds for some $n \geq 2$, rather than for all $n$.

\begin{Lemma}
\label{Lemma:TZA-large-n}
	Let $A$ be a simple unital stably finite infinite dimensional $C^*$-algebra. Suppose that for any $n_0 \in \N$ there exists $k \geq n_0$ such that for any nonzero positive element $a \in A$, for any $\eps>0$ and for any finite set $F \subset A$ there exists an order zero map $\varphi \colon M_k \to A$ such that:
		\begin{enumerate}
		\item $1 - \varphi(1) \precsim a$ ;
		\item For any $x \in M_k$ and for any $y \in F$ we have $\| [\varphi(x),y] \| < \eps \|x\| $ .
		\end{enumerate}
	Then $A$ is tracially $\Zh$-absorbing.
\end{Lemma}
\begin{proof}
The same arguments as in \cite{hirshberg-orovitz} show that the conclusions of \cite[Theorem 3.3]{hirshberg-orovitz} hold under the a-priori weaker hypotheses of this lemma, namely, $A$ has strict comparison. Fix $n \in \N$, $\eps>0$, $a \in A$ positive and non-zero, and $F \subset A$ finite. We want to find an order zero map $\psi \colon M_n \to A$ satisfying the conditions of \ref{def:tZA}. Find two non-zero orthogonal positive contractions $a_0,a_1 \leq a$ (whose existence follows, for example, from \cite[Proposition 4.10]{kirchberg-rordam-non-simple}). Denoting by $QT(A)$ the set of normalized quasitraces on $A$, set $\delta = \min_{\tau \in QT(A)}d_{\tau}(a_0) > 0$. Fix $n_0 \in \N$ such that $n/n_0 < \delta$. Pick $k \geq n_0$ and an order zero map  $\varphi \colon M_k \to A$ satisfying the conditions of the lemma with $a_1$ in place of $a$. Write $k = nl+m$ where $0 \leq m < n$. Let $\pi \colon M_k \to M_n$ be a block diagonal embedding of multiplicity $l$, and let $\psi = \varphi \circ \pi$. For any quasitrace $\tau$, we have $d_{\tau}(\varphi(1_{M_k}) - \psi(1_{M_n}) ) \leq m/k < n/n_0 < \delta$. Thus 
	\begin{align*}
	1-\psi(1_{M_n}) & = \left ( 1-\varphi(1_{M_k}) \right ) +  \left ( \varphi(1_{M_k}) - \psi(1_{M_n}) \right ) \\
	&\precsim \left ( 1-\varphi(1_{M_k}) \right ) \oplus  \left ( \varphi(1_{M_k}) - \psi(1_{M_n}) \right ) \precsim a_0 + a_1 .
	\end{align*}
The condition $\| [\psi(x),y] \| < \eps \|x\| $ for all $x \in M_n$ and for all $y \in F$ follows immediately from the definition.
\end{proof}

\begin{Lemma}
\label{Lemma:two-commuting-order-zero-maps}
	Let $A$ be a $C^*$-algebra. Let $\varphi \colon M_n \to A$ and $\psi \colon M_k \to A$ be two order zero contractions with commuting ranges. Then there exists an order zero map $\theta \colon M_n \otimes M_k \to C^*(\varphi(M_n),\psi(M_k)) \subseteq A$ with $\theta(1) = \varphi(1)\psi(1)$. 
\end{Lemma}
\begin{proof}
	We use the correspondence between order zero maps and homomorphisms from cones; see \cite{winter-zacharias-order-zero}. Let $D$ be the kernel of the canonical homomorphism $CM_n^+ \otimes CM_k^+ \to \C$. We can identify $D$ with the $C^*$-algebra 
	\[ \{f \in C_0([0,1]^2\smallsetminus\{(0,0)\},M_n \otimes M_k) \mid  f([0,1]\times\{0\}) \subseteq M_n \otimes 1_k,  f(\{0\} \times [0,1]) \subseteq 1_n \otimes M_k\} .
	\]
	Define $\varphi_0 \colon M_n \to D$ by $\varphi_0(a)(t,s) = t a \otimes 1_k$ and likewise set define $\psi_0 \colon M_n \to D$ by $\psi_0(a)(t,s) = s 1_n \otimes a$.
	Then there exists a homomorphism $\pi \colon D \to A$ such that $\varphi = \pi \circ \varphi_0$ and $\psi = \pi \circ \psi_0$. Let $h = \varphi_0(1)\psi_0(1)$, that is, $h(t,s) = ts 1_n \otimes 1_k$. Let $J \lhd D$ be the ideal generated by $h$. Note that $J \cong C_0( (0,1]^2) ) \otimes M_n \otimes M_k$. Thus, one can choose an order zero map $\theta_0 \colon M_n \otimes M_k \to J$ with $\theta_0 = h$. The map $\theta = \pi \circ \theta_0$ satisfies the conclusion. 
\end{proof}

\begin{Lemma}
	\label{lemma:perturb-order-zero-map}
	Let $A$ be a $C^*$-algebra, and let $\varphi \colon M_n \to A$ be an order zero map. Then for any $1>\eps>0$ there exists an order zero map $\psi \colon M_n \to A$ such that $1-\psi(1) = \frac{1}{1-\eps}(1-\varphi(1)-\eps)_+$ and $\|\varphi - \psi \|<\eps$. 
\end{Lemma}
\begin{proof}
	Recall (\cite[Corollary 4.2]{winter-zacharias-order-zero}) that there exists a homomorphism $\pi \colon M_n \to A^{**} \cap \varphi(1)'$ such that $\varphi(x) = \pi(x) \varphi(1)$ for all $x \in M_n$. Define $f \colon [0,1] \to [0,1]$ by $f(t) = \max\{1,t/(1-\eps)\}$. Define $\psi(x) = \pi(x) f(\varphi(1))$. It is readily verified that $\psi$ satisfies the required conclusion.
\end{proof}

We recall the following two lemmas.
\begin{Lemma} \cite[Lemma 2.7]{AGP2019}
\label{Lemma:bounded-Cuntz-equivalence}
	Let $A$ be a $C^*$-algebra. Let $a, b \in A_+$, and let $\delta > 0$. If $b \precsim (a-\delta)_+$
	then there exists a sequence $w_n \in A$, $n = 1,2,3,\ldots$ such that $w_n a w_n^* \to b$ and
	$ \| w_n \| \leq \sqrt { \| b \| / \delta  } $ for every $ n \in \N$.
\end{Lemma}

\begin{Lemma} \cite[Lemma 2.2]{kirchberg-rordam-non-simple}
	\label{lemma:bouding-minus-epsilon}
Let $A$ be a $C^*$-algebra. Let $a,b \in A_+$, and let $\eps>0$. If $\|a-b\|<\eps$ then there exists a contraction $d \in A$ such that $dbd^* = (a-\eps)_+$.
\end{Lemma}
 
\begin{Lemma}
	\label{lemma:perturb-positive-element}
	Let $A$ be a unital separable $C^*$-algebra. Let $a \in A_+$ be a nonzero element. Let $b_1,b_2,b_3,\ldots \in A_+$ be a bounded sequence of elements. Let $b$ be the class of the sequence $(b_1,b_2,\ldots)$ in $A_{\infty}$. If $b \precsim a$ then for any $\eps>0$ there exists $n_0$ such that for all $n \geq n_0$ we have $(b_n-\eps)_+ \precsim a$. 
\end{Lemma}
\begin{proof}
	Pick $x \in A_{\infty}$ which satisfies $\| x a x^* - b \| < \eps$. Pick a bounded sequence $(x_1,x_2\ldots ) \in l^{\infty}(A)$ which represents $x$. Then there exists $n_0$ such that for all $n \geq n_0$ we have $\| x_n a x_n^* - b_n \| < \eps$. By Lemma \ref{lemma:bouding-minus-epsilon}, for any such $n$ there exists $y_n$ such that $y_nx_n a x_n^* y_n^* = (b_n-\eps)_+$, as required.
\end{proof}

\begin{Prop}
	\label{prop:conditions-for-TZA}
	Let $A$ be a simple unital infinite dimensional stably finite $C^*$-algebra. Suppose that there exists $n \geq 2$ such that for any nonzero positive element $a \in A$, for any $\eps>0$ and for any finite set $F \subset A$  there exists an order zero map $\varphi \colon M_n \to A$ such that:
		\begin{enumerate}
		\item $1 - \varphi(1) \precsim a$ ;
		\item For any $x \in M_n$ and for any $y \in F$ we have $\| [\varphi(x),y] \| < \eps \|x\| $ .
		\end{enumerate}
	Then $A$ is tracially $\Zh$-absorbing.
\end{Prop} 
\begin{proof}
 Let $n$ as in the statement. We show that the hypothesis holds with $n^2$ in place of $n$, and hence by repeating the argument as many times as needed, the hypothesis of Lemma \ref{Lemma:TZA-large-n} holds.
 
 Fix some $\|a\|> \delta> 0$. The hypotheses of the proposition imply that for any separable subspace $S \subset A_{\infty}$  there exists a sequence of order zero maps $\psi_k \colon M_n \to A$ for $k=1,2,\ldots $ such that $1-\psi_k(1) \precsim (a-\delta)_+$ for all $k$ and such that the composition with the quotient map $l^{\infty}(A) \to A_{\infty}$ defines an order zero map $\psi \colon M_n \to A_{\infty} \cap S'$. By Lemma \ref{Lemma:bounded-Cuntz-equivalence}, we have $1-\psi(1) \precsim a$. (Note that we cannot deduce that $1-\psi(1) \precsim (a-\delta)_+$, as this requires having a sequence of elements implementing the Cuntz subequivalence whose norms are uniformly bounded in $k$.) 
 
 The hypotheses imply that $A$ is not type I, and therefore we can choose two nonzero orthogonal positive elements $a_1,a_2 \leq a$. 
 Fix a finite set $F$ and $\eps>0$. Pick an order zero map $\psi \colon M_n \to A_{\infty} \cap F'$ such that $1-\psi(1) \precsim a_1$. Now, pick another order zero map $\eta \colon M_n \to A_{\infty} \cap (F\cup \psi(M_n) )'$ such that $1-\eta(1) \precsim a_2$. By Lemma \ref{Lemma:two-commuting-order-zero-maps}, there exists an order zero map $\theta \colon M_n \otimes M_n \to C^*(\psi(M_n),\eta(M_n)) \subseteq A_{\infty} \cap F'$ such that $\theta(1) = \psi(1)\eta(1) = \psi(1)^{1/2}\eta(1)\psi(1)^{1/2}$. Then we have
 
 \begin{align*}
 1-\theta(1)  & = ( 1-\psi(1) ) + \psi(1)^{1/2} (1-\eta(1)) \psi(1)^{1/2} \\
 & \precsim  ( 1-\psi(1)  ) \oplus \psi(1)^{1/2} (1-\eta(1)) \psi(1)^{1/2} \\
 & \precsim a_1 \oplus a_2 \\
 & \approx a_1+a_2
 \end{align*}
 We lift $\theta$ to sequence of order zero maps $\theta_k \colon M_n \otimes M_n \to A$, for $k=1,2,\ldots$. Pick a sufficiently large $k$ such that $\| [\theta_k(x),y ] \| < \eps \|x\|/2$ for all $x \in M_n \otimes M_n$ and for all $y \in F$ and such that $(1-\theta_k(1) - \eps/2)_+ \precsim a$. Using Lemmas \ref{lemma:perturb-order-zero-map} and \ref{lemma:perturb-positive-element}, we can find an order zero map $\theta' \colon M_n \otimes M_n \to A$ such that $\| \theta' - \theta_k \| < \eps/2$ and $1-\theta'(1) = \frac{1}{1-\eps/2}(1-\theta_k(1)-\eps/2)_+$. The map $\theta'$ is now readily verified to satisfy the conclusion.
\end{proof}
 
 Using stability of generators of matrix cones (\cite{loring}), we have the following reformulation of Proposition \ref{prop:conditions-for-TZA}.
\begin{Lemma}
\label{lemma:TZA-stability-relations}
Let $A$ be a simple unital infinite dimensional stably finite $C^*$-algebra. Suppose that any nonzero positive element $a \in A$, for any $\eps>0$ and for any finite set $F \subset A$ there exists $n \geq 2$ contractions $x_2,x_3,\ldots, x_n$ which satisfy 
		\begin{enumerate}
		\item $\|x_jx_k\|<\eps$ for all $j,k \in \{2,3,\ldots,n\}$;
		\item $\|x_j^*x_k\|<\eps$ whenever $k \neq j$;
		\item $\|x_j^*x_j - x_k^*x_k\|< \eps$ for all $j,k \in \{2,3,\ldots,n\}$;
		\item There exists a positive contraction $b \precsim a$ such that $\| 1 - x_2^*x_2 - \sum_{k=2}^n x_k x_k^* - b \| < \eps $ ;
		\item For any $k=2,3,\ldots,n$ and for any $y \in F$ we have $\| [x_k,y] \| < \eps $ .
		\end{enumerate}
	Then $A$ is tracially $\Zh$-absorbing.
\end{Lemma}
 
We recall the definition of a strongly tracially approximately inner automorphism, introduced in \cite{phillips-tracial}. This is a significant weakening of the notion of approximate innerness. For example, automorphisms with this property need not act trivially on $K$-theory; see \cite[Theorem 6.6]{phillips-tracial}.

\begin{Def}\cite[Definition 4.2]{phillips-tracial}
\label{def:STAI}
Let $A$ be an infinite dimensional simple separable unital $C^*$-algebra  and  let
$\alpha \in \aut (A)$.  We  say  that
$\alpha$ is \emph{strongly tracially approximately inner} if for every finite set
$F \subseteq A$, 
every $\eps>0$, and every positive element
$x\in A$ with
$\|x\|= 1$, there exist a projection
$p \in A$ 
and a partial isometry
$v\in A$ 
such that:
\begin{enumerate}
\item  $v^*v=p$ and $vv^*=\alpha(p)$.
\item $ \| pa - ap \| < \eps$ for all
$a\in F$.
\item $ \| vpapv^* - \alpha (pap) \| < \eps$ 
for all $a\in F$.
\item $1-p$ is Murray-von Neumann equivalent to a projection in $xAx$.
\item $\| pxp \| >1-\eps$.
\end{enumerate}
\end{Def}
Notice that in Definition \ref{def:STAI}, it follows that  $ \| pap - v^* \alpha (pap) v \| < \eps$ for all $a \in F$ as well, and $v = \alpha(p)v = vp$. Furthermore, for any $a \in F$, 
$va = vpa \approx_{\eps} v pap = v pap v^* v \approx_{\eps} \alpha(pap) v \approx_{\eps} \alpha (ap) v = \alpha(a) v$.

We now recall the definition of the generalized tracial Rokhlin property for finite group actions and for actions of $\Z$ (\cite[Definitions 5.2 and 6.1]{hirshberg-orovitz}).

\begin{Def}
Let $\alpha \colon G \to \aut(A)$ be an action of a finite group $G$ on a simple, unital $C^*$-algebra $A$. The action
$\alpha$ is said to have the \emph{generalized tracial Rokhlin property} if for any 
$\eps > 0$, for any
finite subset 
$F \subseteq A$, and for any non-zero positive element 
$a \in A$ there exist positive
contractions 
$\{e_g\}_{g\in G} \subset A$ of norm $1$ such that:
\begin{enumerate}
\item  $e_g e_h = 0$ whenever $ g \neq h$;
\item $1 - \sum_{g\in G} e_g \precsim a$;
\item 	$ \| [e_g, y] \|< \eps$ for all $g \in G$ and for all $ y \in F$;
\item $ \|	\alpha_g(e_h) - e_{gh} \| < \eps$ for all $g,h \in G$.
\end{enumerate}
\end{Def}

\begin{Def}
\label{def:gtRp-Z}
 Let $\alpha$ be an automorphism of a simple, unital $C^*$-algebra $A$. We say that $\alpha$ has
the \emph{generalized tracial Rokhlin property} if for any finite set $F \subseteq A$, for any $\eps > 0$, for any 
$k \in \N$, and for
any non-zero positive element 
$a \in A$ there exist orthogonal positive norm $1$ contractions
$e_1, e_2 \ldots , e_k \in A$ such that the following hold:
\begin{enumerate}
\item $1 - \sum_{j=1}^k e_j \precsim a$;
\item $ \| 	[e_j, y] \|	 < \eps$ for $j=1,2,\ldots,k$ and for all $y \in F$;
\item $ \|	\alpha(e_j ) - e_{j+1} \|	< \eps$ for all $j=1,2,\ldots,k-1$.
\end{enumerate}
\end{Def}

\begin{Lemma}
\label{lemma:gTRP-restrict-to-subgroup}
Let $\alpha \colon G \to A$ be an action of a group $G$ on a simple, unital $C^*$-algebra $A$ with $G$ finite or $G = \Z$. If $\alpha$ has the generalized tracial Rokhlin property then so does any restriction to a subgroup.
\end{Lemma}
The proof is straightforward and similar to that of Lemma \ref{Lemma:restrict-to-subgroup}, so we omit it. 

\begin{Thm}
	\label{thm:tracially-Z-abs}
	Let $A$ be a simple unital infinite dimensional stably finite $C^*$-algebra. Let $G$ be a finite group or $G = \Z$. Let $\alpha \colon G \to \aut(A)$ be an action which has the generalized tracial Rokhlin property and such that $\alpha_g$ is strongly tracially approximately inner for some nontrivial $g \in G$. If $A$ furthermore has stable rank 1, then $A$ is tracially $\Zh$-absorbing. If $\alpha_g$ above is in fact approximately inner, then the stable rank 1 condition is not needed.
\end{Thm}

\begin{proof}
		Using Lemma \ref{lemma:gTRP-restrict-to-subgroup}, we may assume that $G$ is a cyclic group (of either finite or infinite order), and the generator $\alpha$ is  strongly tracially approximately inner. Let $n$ be the order of $\alpha$ if $G$ is finite, and otherwise fix any integer $n \geq 2$.
	
	Fix a nonzero positive elements $a \in A$, a finite set $F \subset A$ and $\eps>0$. We show that there exists an order zero map $\varphi \colon M_n \to A$ satisfying the hypotheses of Proposition \ref{prop:conditions-for-TZA}.
	
	By \cite[Proposition 4.10]{kirchberg-rordam-non-simple}, there exist pairwise orthogonal and mutually Cuntz-equivalent positive contractions $a_0,a_1,\ldots,a_{2n} \precsim a$, which we may assume have norm $1$.

	Set $\eps' = \eps/6n$. 		
	Choose a partial isometry $v$ as in Definition \ref{def:STAI} for the given finite set $F$ with $a_0$ in place of $x$ and $\eps'$ in place of $\eps$.  Set $p = v^*v$.
	
	Now, choose positive elements $e_1,e_2,\ldots,e_n \in A$ as in Definition \ref{def:gtRp-Z} with $F \cup \{\alpha^j(v),\alpha^j(p)\}_{j=0,1,2,\ldots,n-1}$ in place of $F$, $\eps'$ in place of $\eps$ and $a_0$ in place of $a$. If $G = \Z_n$, then we use the same notation, where we can think of $\{1,2,\ldots,n\}$ as indexing the group elements. We furthermore choose the Rokhlin elements $e_1,e_2,\ldots,e_n$ to satisfy $\|\alpha (e_j^{1/2}) - e_{j+1}^{1/2} \| < \eps'$ for $j=1,2,\ldots,n-1$ as well as $\|[e_j^{1/2},y]\|<\eps'$ and $\|[e_j^{1/4},y]\|<\eps'$ for $j=1,2,\ldots,n$ and for all $y \in F$ (which can be done in a standard way by approximating $e_j^{1/2}$ and $e_j^{1/4}$ by polynomials). Set $e = \sum_{j=1}^n e_j$. 
	
	Lastly, we choose a partial isometry $w$ (in place of $v$) as in Definition \ref{def:STAI} with the set 
	\[
	F \cup \{\alpha^j(v),\alpha^j(v^*),\alpha^j(p)\}_{j=0,1,2,\ldots,n-1} \cup \{e_j,e_j^{1/2},e_j^{1/4}\}_{j=1,2,\ldots,n} 
	\]  in place of $F$, with $a_0$ in place of $x$ and with $\eps'$ in place of $\eps$.  Set $q = w^*w$.
	
	For $k=2,3,\ldots,n$, set 
	\[
	x_k =  e_k^{1/4}  v^* \alpha (v^*) \alpha^2(v^*) \cdots \alpha^{k-2}(v^*) \alpha^{k-2}(w) \cdots \alpha^2 (w) \alpha (w) w  e_1^{1/4}
	.
	\]
	We check that the elements $x_1,x_2,\ldots,x_n$ satisfy the conditions of Lemma \ref{lemma:TZA-stability-relations}.
	Using repeatedly the fact that $vp = v$ and $vv^* = \alpha(p)$ and likewise that $wq = w$ and $ww^* = \alpha(q)$ we have
	\begin{align*}
		x_k^* x_k &=  e_1^{1/4} w^* \alpha(w^*) \cdots \alpha^{k-2}(w^*) \alpha^{k-2}(v) \cdots \alpha(v) v e_k^{1/2} v^* \alpha (v^*)  \cdots \alpha^{k-2}(v^*) \alpha^{k-2}(w) \cdots \alpha (w) w  e_1^{1/4} \\
			&\approx_{(k-1)\eps'}  e_1^{1/4} w^*  \cdots \alpha^{k-2}(w^*)  e_k^{1/2} \alpha^{k-2}(v) \cdots \alpha(v) v v^* \alpha (v^*)  \cdots \alpha^{k-2}(v^*) \alpha^{k-2}(w) \cdots  w  e_1^{1/4}  \\
			&=   e_1^{1/4} w^*  \cdots \alpha^{k-2}(w^*)  e_k^{1/2} \alpha^{k-2}(v) \cdots \alpha(v) \alpha(p) \alpha (v^*)  \cdots \alpha^{k-2}(v^*) \alpha^{k-2}(w) \cdots w  e_1^{1/4} \\
			&=   e_1^{1/4} w^*  \cdots \alpha^{k-2}(w^*)  e_k^{1/2} \alpha^{k-2}(v) \cdots \alpha(v) \alpha (v^*)  \cdots \alpha^{k-2}(v^*) \alpha^{k-2}(w) \cdots  w  e_1^{1/4}\\
			&=   e_1^{1/4} w^*  \cdots \alpha^{k-2}(w^*)  e_k^{1/2} \alpha^{k-2}(v) \cdots \alpha^2(v) \alpha(vv^*) \alpha^2 (v^*)  \cdots \alpha^{k-2}(v^*) \alpha^{k-2}(w) \cdots  w  e_1^{1/4}\\
			&=   e_1^{1/4} w^*\cdots \alpha^{k-2}(w^*)  e_k^{1/2} \alpha^{k-2}(v) \cdots \alpha^2(v) \alpha^2(p) \alpha^2 (v^*)  \cdots \alpha^{k-2}(v^*) \alpha^{k-2}(w) \cdots  w  e_1^{1/4}\\
			&=   e_1^{1/4} w^* \cdots \alpha^{k-2}(w^*)  e_k^{1/2} \alpha^{k-2}(v) \cdots \alpha^2(v^*)  \cdots \alpha^{k-2}(v^*) \alpha^{k-2}(w) \cdots  w  e_1^{1/4}\\
			& \vdots \\
			&=   e_1^{1/4} w^* \alpha(w^*) \cdots \alpha^{k-2}(w^*)  e_k^{1/2} \alpha^{k-1}(p) \alpha^{k-2}(w) \cdots \alpha (w) w  e_1^{1/4}\\
			&\approx_{(k-1)\eps'}   e_1^{1/4} w^* \alpha(w^*) \cdots \alpha^{k-2}(w^* \alpha (e_1^{1/2} p)  w) \cdots \alpha (w) w  e_1^{1/4} \\
			&\approx_{\eps'}   e_1^{1/4} w^* \alpha(w^*) \cdots \alpha^{k-2}(q e_1^{1/2} p q) \cdots \alpha (w) w  e_1^{1/4} \\
			\\ & \vdots \\
			&\approx_{\eps'} e_1^{1/4} q e_1^{1/2} p q e_1^{1/4} \\
			&\approx_{5\eps'} e_1^{1/2} q p q e_1^{1/2} 
	\end{align*}
	so $\|	x_k^* x_k - e_1^{1/2} q p q e_1^{1/2}   \| < ( 3(k-1)+5 ) \eps'$. 
	Likewise, using at the last step repeatedly that $v^*\alpha(v^*v)v = v^* v v^* v = p$, we have
	\begin{align*}
			x_k x_k^* &=  e_k^{1/4} v^* \alpha (v^*)  \cdots \alpha^{k-2}(v^*) \alpha^{k-2}(w) \cdots \alpha (w) w  e_1^{1/2} w^* \alpha(w^*) \cdots \alpha^{k-2}(w^*) \alpha^{k-2}(v) \cdots \alpha(v) v  e_k^{1/4} \\
			&\approx_{\eps'} e_k^{1/4} v^* \alpha (v^*)  \cdots \alpha^{k-2}(v^*) \alpha^{k-2}(w) \cdots \alpha(w) \alpha (qe_1^{1/2}q) \alpha(w^*) \cdots \alpha^{k-2}(w^*) \alpha^{k-2}(v) \cdots \alpha(v) v  e_k^{1/4} \\
			&=  e_k^{1/4} v^* \alpha (v^*)  \cdots \alpha^{k-2}(v^*) \alpha^{k-2}(w) \cdots \alpha(w e_1^{1/2} w^*) \cdots \alpha^{k-2}(w^*) \alpha^{k-2}(v) \cdots \alpha(v) v  e_k^{1/4} \\
			&\approx_{\eps'}  e_k^{1/4} v^* \alpha (v^*)  \cdots \alpha^{k-2}(v^*) \alpha^{k-2}(w) \cdots \alpha^2(q e_1^{1/2} q) \cdots \alpha^{k-2}(w^*) \alpha^{k-2}(v) \cdots \alpha(v) v  e_k^{1/4} \\
			& \vdots \\
			&\approx_{\eps'}  e_k^{1/4} v^* \alpha (v^*)  \cdots \alpha^{k-2}(v^*) \alpha^{k-1}(q e_1^{1/2} q) \alpha^{k-2}(v) \cdots \alpha(v) v  e_k^{1/4} \\
			&\approx_{\eps'}  e_k^{1/4} v^* \alpha (v^*)  \cdots \alpha^{k-2}(v^*) \alpha^{k-1}(q) e_k^{1/2} \alpha^{k-1}(q) \alpha^{k-2}(v) \cdots \alpha(v) v  e_k^{1/4} \\
			&\approx_{4k\eps'}  e_k^{1/2} \alpha^{k-1}(q) v^* \alpha (v^*)  \cdots \alpha^{k-2}(v^*) \alpha^{k-2}(v) \cdots \alpha(v) v \alpha^{k-1}(q)  e_k^{1/2} \\
			&=  e_k^{1/2} \alpha^{k-1}(q) p \alpha^{k-1}(q)  e_k^{1/2} \\
	\end{align*}
	so $ \| x_k x_k^* - e_k^{1/2} \alpha^{k-1}(q) p \alpha^{k-1}(q)  e_k^{1/2} \| < 5k\eps' $.
	
	Now, notice that $x_jx_k = 0$ for all $j,k$, and $x_j^*x_k = 0$ whenever $j \neq k$. 
	
	If $y \in F$ and $k \in \{2,3,\ldots,n\}$, we have
	\begin{align*}
	y x_k &= y e_k^{1/4} v^* \alpha (v^*)  \cdots \alpha^{k-2}(v^*) \alpha^{k-2}(w) \cdots \alpha (w) w  e_1^{1/4} \\
	&\approx_{\eps'}  e_k^{1/4} y v^* \alpha (v^*)  \cdots \alpha^{k-2}(v^*) \alpha^{k-2}(w) \cdots \alpha (w) w  e_1^{1/4} \\
	&\approx_{3\eps'}  e_k^{1/4}  v^* \alpha (y) \alpha (v^*)  \cdots \alpha^{k-2}(v^*) \alpha^{k-2}(w) \cdots \alpha (w) w  e_1^{1/4} \\
	&\approx_{3\eps'}  e_k^{1/4}  v^*  \alpha (v^*) \alpha^2 (y) \cdots \alpha^{k-2}(v^*) \alpha^{k-2}(w) \cdots \alpha (w) w  e_1^{1/4} \\
	& \vdots \\
	&\approx_{3\eps'}  e_k^{1/4}  v^*  \alpha (v^*)  \cdots \alpha^{k-2}(v^*) \alpha^{k-1}(y) \alpha^{k-2}(w) \cdots \alpha (w) w  e_1^{1/4} \\
	&=  e_k^{1/4}  v^*  \alpha (v^*)  \cdots \alpha^{k-2}(v^*) \alpha^{k-2}(\alpha(y) w) \cdots \alpha (w) w  e_1^{1/4} \\
	&\approx_{3\eps'}  e_k^{1/4}  v^*  \alpha (v^*)  \cdots \alpha^{k-2}(v^*) \alpha^{k-2}( w y) \cdots \alpha (w) w  e_1^{1/4} \\
	& \vdots \\
	&\approx_{3\eps'}  e_k^{1/4}  v^*  \alpha (v^*)  \cdots \alpha^{k-2}(v^*) \alpha^{k-2}( w ) \cdots \alpha (w) w y e_1^{1/4} \\
	&\approx_{\eps'}  e_k^{1/4}  v^*  \alpha (v^*)  \cdots \alpha^{k-2}(v^*) \alpha^{k-2}( w ) \cdots \alpha (w) w  e_1^{1/4} y \\
	&= x_k y
	\end{align*}
	so $\| y x_k - x_k y \| < 6(k-1)\eps' + 2\eps'$.
	
	Since $A$ has stable rank 1, the projection $1-q = 1-w^*w$ is Murray - von Neumann equivalent to $1- \alpha(q) = 1 - ww^*$, and likewise to $1-\alpha^k(q)$ for $k=1,2,\ldots,n-1$. This is the only place in which we need the assumption that $A$ has stable rank 1; if $\alpha$ is approximately inner, then one could take $q = 1$, in which case all the projections of the form $1-\alpha^k(q)$ are zero, hence equivalent trivially. (One could have replaced stable rank 1 in the hypothesis of the theorem by a strengthening of the notion of strong tracial approximate innerness which asks that $1-q$ be Murray - von Neumann equivalent to $1-\alpha(q)$.)
	So
	\[
	\left \|\left ( 1 - x_2^*x_2 - \sum_{k=2}^n x_k x_k^* \right ) - \left (  1 - \sum_{k=1}^{n} e_k^{1/2} \alpha^{k-1}(q) p \alpha^{k-1}(q)  e_k^{1/2} \right ) \right  \| < \eps 
	.
	\]
	and
	\begin{align*}
	 1 - \sum_{k=1}^{n} e_k^{1/2} \alpha^{k-1}(q) p \alpha^{k-1}(q)  e_k^{1/2}  
	&= \left ( 1 - \sum_{k=1}^{n} e_k \right )  +
	\sum_{k=1}^{n} e_k^{1/2} \left ( 1 - \alpha^{k-1}(q) \right )  e_k^{1/2}  + \\
	& \quad +
	\sum_{k=1}^{n} e_k^{1/2} \left ( \alpha^{k-1}(q) (1-p) \alpha^{k-1}(q) \right )  e_k^{1/2} \\ 
	& \precsim \left ( 1 - \sum_{k=1}^{n} e_k \right ) \oplus 
		\bigoplus_{k=1}^{n} e_k^{1/2} \left ( 1 - \alpha^{k-1}(q) \right )  e_k^{1/2}   \\
		& \quad \oplus
		\bigoplus_{k=1}^{n} e_k^{1/2} \left ( \alpha^{k-1}(q) (1-p) \alpha^{k-1}(q) \right )  e_k^{1/2} 
	\end{align*}
	Thus, the class of  $1 - \sum_{k=1}^{n} e_k^{1/2} \alpha^{k-1}(q) p \alpha^{k-1}(q)  e_k^{1/2}$ is bounded in the Cuntz semigroup by a direct sum of $2n+1$ elements each of which is Cuntz-subequivalent to $a_0$, and therefore is bounded by the class of $a$.  
	
	It follows that $A$ is tracially $\Zh$-absorbing, as required.
\end{proof}

\begin{Rmk}
	If in Theorem \ref{thm:tracially-Z-abs} we assume furthermore that $A$ is nuclear, then it follows from \cite[Theorem 4.1]{hirshberg-orovitz} that $A$ is $\Zh$-absorbing. This addresses in part \cite[Question 7.2]{AGP2019}, which asks whether there exists a simple separable unital nuclear $C^*$-algebras with positive radius of comparison which admits a finite group action with the weak tracial Rokhlin property but the action on the trace space is trivial. One might try to look for approximately inner examples, or even strongly tracially approximately inner ones, but we see here that those do not exist. This does not answer the question, even if we furthermore require that the action fixes the entire Elliott invariant, since uniqueness and existence for automorphisms break down in the case of positive radius of comparison (\cite{toms-counterexample,hirshberg-phillips-asymmetry}); however it shows that constructing such an example may be complicated.
\end{Rmk}

\section{Rokhlin dimension and divisibility conditions}
\label{section Rokhlin dimension and divisibility conditions}

Without the approximate innerness condition, there are many examples of non-$\mathcal{Z}$-stable $C^*$-algebras which admit actions with finite Rokhlin dimension with commuting towers. The obvious examples are commutative $C^*$-algebras: if $X$ is a compact metric space with finite covering dimension, and $G$ is a finitely generated nilpotent group acting freely in $X$, then the associated action $\alpha \colon G \to \aut(C(X))$ has finite Rokhlin dimension with commuting towers (\cite[Corollary 8.5]{SWZ}). In the simple nuclear setting, it was shown in \cite{AGP2019} that there exists a non $\Zh$-stable simple unital separable AH algebra along with an action of $\Z_2$ which has the Rokhlin property. Furthermore, the crossed product need not be $\Zh$-stable either: one can construct examples, along the lines of \cite{giol-kerr}, of minimal homeomorphisms of infinite dimensional compact metric spaces $X$ which have a Cantor minimal system as a factor and such that the crossed product is not $\Zh$-stable; such actions have the Rokhlin property. 

However, finite Rokhlin dimension with commuting towers is sufficient to deduce that the crossed product has a weaker regularity property. In \cite{DHTW}, based on ideas which appeared in \cite{HRW}, we constructed an example of a simple unital separable AH algebra which does admit a unital embedding of the Jiang-Su algebra, or in fact, even an embedding of the dimension drop algebra $I_{3,4}$. This type of phenomenon was studied in greater depth in \cite{Robert-Rordam}. 
We recall the definition.  We denote by $\Cu(A)$ the Cuntz semigroup of a $C^*$-algebra $A$.
\begin{Def} (Robert-R{\o}rdam)
	Let $A$ be a unital $C^*$-algebra. Let $m$ be a positive integer. We define $\Div_m(A)$ to be the least integer $n$ for which there exists $x \in \Cu(A)$ such that $m x \leq \left < 1_A \right > \leq nx$ (where we set $\Div_m(A)=\infty$ if no such $n$ exists).
\end{Def}
If there exists a unital homomorphism from $I_{m,m+1}$  into $A$ then $\Div_m(A) \leq m+1$, and the converse holds if $A$ has stable rank 1. (See \cite[Example 3.12]{Robert-Rordam}.) 

 The following shows that examples with such bad divisibility properties cannot occur as crossed products when the actions has finite Rokhlin dimension with commuting towers, provided the group has elements of arbitrarily large order.
\begin{Thm}
		Let $A$ be a separable unital $C^*$-algebra. Let $G$ be a residually finite group which has elements of arbitrarily large order. Let $\alpha \colon G \to \aut(A)$ be an action which has finite Rokhlin dimension with commuting towers. Then for any $m = 2,3,\ldots$, there is a unital homomorphism from $I_{m,m+1}$ into $A \rtimes_{\alpha} G$. In particular, $\Div_m(A \rtimes_{\alpha} G) \leq m+1$ for all $m \geq 2$.
\end{Thm}
\begin{proof}
	By \cite[Theorem 6.2]{Dadarlat-Toms}, there exists a constant $L$ such that for any compact metric space $Y$ and for any coprime numbers $p,q$, if $n \geq L(pq)^2(\dim(Y)+2)$ and $A$ is a locally trivial $n$-homogeneous algebra with spectrum $Y$ then there exists a unital homomorphism from $I_{p,q}$ to $A$. (The hypotheses of \cite[Theorem 6.2]{Dadarlat-Toms} are more general, as it also covers recursive subhomogeneous algebras, but we do not need this generality here.) Since the dimension drop algebra $I_{m,m+1}$ is semiprojective, it suffices to show that it admits a unital homomorphism into the ultrapower $(A \rtimes_{\alpha} G)_{\omega}$.
	
	Set $d = \dimrokc(\alpha)$.  Given $m \geq 2$, and with the constant $L$ as above, fix $n \geq L m^2 (m+1)^2(2d+4)$. Pick an element $g_0 \in G$ with order at least $n$. Set $H = \left < g_0 \right >$. We know that $\dimrokc(\alpha|_H) \leq 2d+1$.
	
	Because $A \rtimes_{\alpha|_{H}} H \subseteq A \rtimes_{\alpha} G$, it suffices to show that there exists a unital homomorphism from $I_{m,m+1}$ into $(A \rtimes_{\alpha|_{H}} H)_{\omega}$.

	Let $u = u_{g_0}$ be the canonical unitary in $A \rtimes_{\alpha} H$ (which we identify with its image in the ultrapower). Pick an equivariant unital homomorphism $\psi \colon C(X_{\Z,n\Z,2d+1}) \to A_{\omega} \cap A'$. (The fact that the image is in the relative commutant of $A$ does not play a role here.) Because $u\varphi (f) u^* = \sigma(f)$ for all $f \in  C(X_{\Z,n\Z,2d+1})$, the pair $(\varphi,u)$ defines a unital homomorphism from $C(X_{\Z,n\Z,2d+1}) \rtimes_{\sigma} \Z$ to  $A_{\omega} \rtimes_{\alpha} H \subseteq (A \rtimes_{\alpha} H)_{\omega}$. As $C(X_{\Z,n\Z,2d+1}) \rtimes_{\sigma} \Z$ is $n$-homogeneous over the $(2d+2)$-dimensional space $X_{\Z,n\Z,2d+1}/\Z \times \T$, it admits a unital homomorphism from $I_{m,m+1}$, and therefore so does $(A \rtimes_{\alpha} H)_{\omega}$, as required.
\end{proof}

\section{Nearly approximately inner automorphisms}
\label{section Nearly approximately inner automorphisms}
In this section, we introduce a notion which is weaker than approximate innerness, and is sufficient for deducing that the crossed product is $\Zh$-stable (even when the $C^*$-algebra acted on is not). We restrict our attention just to crossed products by $\Z$ with the Rokhlin property. It seems plausible that those results should hold for more general groups and for finite Rokhlin dimension with commuting towers.

\begin{Def}
	\label{Def: nearly approximately inner}
	Let $A$ be a unital $C^*$-algebra, and let $\alpha \in \aut(A)$. We say that $\alpha$ is \emph{nearly approximately inner} if there exists a dense subset $A_0$ of the unit ball of $A$ and a sequence $n_l \to \infty$ of positive integers such that for any $\eps>0$ and any finite $F \subset A_0$ there exists $l_0 \in \N$ such that for all $l \geq l_0$, there exist $v_l \in U(A)$ such that $\|v_la v_l^*  - \alpha^{n_l}(a)\| < \eps$ for all $a \in \alpha^{k}(F)$, for $k \in \{-7 n_l, -7 n_l +1 \ldots , 7 n_l \}$.
\end{Def}
The choice of the factor $7$ in the definition is of course artificial, and simply comes up as a technical detail in the proof of Theorem \ref{Thm:NAI-Z-stable} (where no serious attempt was made to optimize the constant).  We mention two stronger conditions which may seem more natural more natural, each of which implies the technical condition in Definition \ref{Def: nearly approximately inner}.

\begin{Prop}
	\label{Prop: conditions which imply NAI}
	Let $A$ be a unital $C^*$-algebra, and let $\alpha \in \aut(A)$. Each of the following conditions implies that $\alpha$ is nearly approximately inner.
		\begin{enumerate}
			\item There exists $n > 0$ such that $\alpha^n$ is approximately inner.
			\item 
			\label{cond:AP}
			There exists a dense subset $A_0$ of the unit ball of $A$ and a sequence $n_l \to \infty$ of positive integers such that for any $\eps>0$ and any finite $F \subset A_0$ there exists $l_0 \in \N$ such that for all $l \geq l_0$, there exist $v_l \in U(A)$ such that $v_l\alpha(v_l)^* \in Z(A)$ and $\|v_la v_l^*  - \alpha^{n_l}(a)\| < \eps$ for all $a \in F$.
		\end{enumerate}
\end{Prop} 
	\begin{proof}
		If there exists $n>0$ such that $\alpha^n$ is approximately inner, the definition is easily seen to be satisfied by setting $n_l = n\cdot l$.
		
		For the second condition, it suffices to check that if $v_l\alpha(v_l)^*$ is in the center of $A$, and $\|v_la v_l^*  - \alpha^{n_l}(a)\| < \eps$ for all $a \in F$, then in fact $\|v_la v_l^*  - \alpha^{n_l}(a)\| < \eps$ for any $a \in \alpha^k(F)$ for any $k \in \Z$. Notice that for any $x \in F$, because $(v_l^* \alpha(v_l)) x  = x (v_l^* \alpha(v_l))$, we have
		$$
		\alpha(v_l) x \alpha(v_l)^* = v_l (v_l^* \alpha(v_l)) x (\alpha(v_l)^* v_l) v_l^* = v_l x v_l
		$$ 
		Thus, if $a = \alpha^k(b)$ for $b \in F$ then
		$$
		\|v_l a v_l^*-  \alpha^{n_l}(a)\| = \| \alpha^k(v_l) a \alpha^k(v_l^*)  - \alpha^{n_l}(a)\|= \| \alpha^k(v_l b v_l^*  - \alpha^{n_l}(b))\| < \eps
		$$
		as required.
	\end{proof}
We now provide some interesting examples of automorphisms which are nearly approximately inner, but such that no nontrivial power is approximately inner. In fact, our examples satisfy condition \ref{cond:AP} from Proposition \ref{Prop: conditions which imply NAI}, with $v = 1$. Those can all be viewed as generalized odometers, though acting on noncommutative $C^*$-algebras. The most interesting of those are perhaps ones which act on the reduced group $C^*$-algebras of the Dixmier-Lance groups, whose group factors have property $\Gamma$ but are not McDuff.
\begin{Exl} We provide here three classes of increasingly complicated examples.
	\begin{enumerate}
	\item  The simplest example we consider arises from a translation on a compact group. Suppose $G$ is a compact metrizable group. Recall that by a theorem of Birkhoff and Kakutani (\cite[Theorem 8.3]{Hewitt-Ross}), $G$ admits a left invariant metric. Fix $g \in G$ of infinite order and set $\alpha \in \aut(C(G))$ by $\alpha(f)(h) = f(g^{-1}h)$. Then one readily checks that $\alpha$ satisfies condition \ref{cond:AP} of Proposition \ref{Prop: conditions which imply NAI} (where we take $A_0$ to be a dense set consisting of Lipschitz functions with respect to an invariant metric and $v = 1$). The odometer action arises when $G$ is a compact abelian group homeomorphic to the Cantor set, and the subgroup generated by $g$ is dense. In this case, the action furthermore satisfies the Rokhlin property. The crossed product by an odometer action is well understood from the perspective of Elliott's classification program, and in particular it is known to be $\Zh$-stable; however, we use this as a basis for constructing more complicated examples in which our methods give genuinely new results, to the best of our knowledge. To this end, let us present this example in an equivalent way. Suppose $G$ is a countable discrete abelian group. Let $\gamma \in \widehat{G}$ be an element of infinite order. Then $C^*(G) \cong C(\widehat{G})$ and if we denote by $\{u_g \mid g \in G\}$ the generators of $C^*(G)$, then the automorphism $\alpha \in \aut(C^*(G))$ given by translation by $\gamma$ on $C(\widehat{G})$ is given by $\alpha(u_g) = \gamma(g)u_g$. If we think of $C^*(G) \cong C^*_r(G)$ as being represented on $l^2(G)$ via the regular representation, then $\alpha(x) = UxU^*$, where $U \in B(l^2(G))$ is the diagonal operator given on the standard basis $\{\delta_g\}_{g \in G}$ by $U\delta_g = \gamma(g)\delta_g$. In this picture, it may be more convenient to take $A_0$ to be a dense subset of the unit ball of the group algebra $\C G$.
	\item Suppose $G$ is a countable discrete abelian group. Let $\gamma \in \widehat{G}$ be an element of infinite order. Let $\varphi \colon G*G \to G$ be the canonical quotient given by identifying the two copies of $G$. 	Consider the action $\alpha$ on $C^*_r(G * G)$ given by $\alpha (u_h) = \gamma(\varphi(h))u_h$ for any $h \in G$. (This is well defined on the reduced group $C^*$-algebra, as it is also given by conjugation by a diagonal unitary $U \in B(l^2(G * G))$.) To give an example which furthermore has the Rokhlin property, assume that $\widehat{G}$ is homeomorphic to the Cantor set, and in place of $C^*_r(G * G)$, consider $C_r^*(G) \otimes C^*_r(G * G) \cong C^*_r(G \times (G*G))$ (the tensor product is unique, as $C^*_r(G)$ is abelian), with the action $\alpha$ defined in a similar fashion, noting that there is also a canonical quotient $G \times (G * G) \to G$. The restriction to $C^*_r(G) \otimes 1_{ C^*_r(G * G) }$ has the Rokhlin property, and since this is in the center of $C^*_r(G) \otimes C^*_r(G * G)$, so does $\alpha$. 
	\item In \cite{Dixmier-Lance}, Dixmier and Lance constructed examples of ICC groups $K$ such that the group von Neumann algebra $L(K)$ has property $\Gamma$ but is not McDuff (that is, $L(K) \not \cong L(K) \overline{\otimes}R$ when $R$ is the hyperfinite II$_1$ factor; or as it is phrased in \cite{Dixmier-Lance}, there are nontrivial central sequences, and all central sequences are hypercentral). In particular, for any such group, the reduced group $C^*$-algebra $C^*_r(K)$ is not $\Zh$-stable (and is not nuclear). The class of $C^*$-algebras of the form $C^*_r(K)$ is thus an interesting set of examples to consider. Those groups are constructed as follows. Let $H$ be a nontrivial countable abelian group. Let $G_1,G_2,\ldots$ and $H_1,H_2,\ldots$ be sequences of groups with $G_k \cong \Z$ for all $k \in \N$ and $H_k \cong H$ for all $k \in \N$. We let $K$ be the group generated by copies of all of those groups subject to the following relations:
		\begin{enumerate}
			\item The subgroup $H_j$ commutes with the subgroup $H_k$ for all $j \neq k$.
			\item Whenever $k \leq j$, the subgroup $G_k$ commutes with the subgroup $H_j$.
		\end{enumerate}
	We now construct an automorphism of $C^*_r(K)$ for a Dixmier-Lance group which satisfies condition \ref{cond:AP} from Proposition \ref{Prop: conditions which imply NAI} and furthermore has the Rokhlin property. We choose an abelian group $H$ such that $\widehat{H}$ is homeomorphic to the Cantor set and has an element $\gamma \in \widehat{H}$ which generates a dense subgroup. We have a canonical surjective homomorphism  $\varphi \colon K \to H$ given by sending all the subgroups $G_k$ to the trivial element and identifying all copies of $H$. We define $\alpha$ by $\alpha(u_g) = \gamma (\varphi(g)) u_g$, which uniquely extends to $C^*_r(K)$ since it is given by $\alpha(x) = UxU^*$ for $U \in B(l^2(K))$ defined by $U \delta_g = \gamma (\varphi(g)) \delta_g$. One verifies that condition \ref{cond:AP} from Proposition \ref{Prop: conditions which imply NAI} holds as before. As for the Rokhlin property, we note that for any $k \in \N$, the restriction of $\alpha$ to the invariant commutative subalgebra $C^*_r(H_k) \subset C^*_r(K)$ has the Rokhlin property, being an odometer action. For any $n \in \N$ and any $k \in \N$, let $\{e^{(k)}_{0,j}\}_{j=0,1,\ldots,n-1} \cup \{ e^{(k)}_{1,j}\}_{j=0,1,\ldots,n}$ be Rokhlin towers in $C^*_r(H_k)$. Those form a central sequence of Rokhlin towers in $C^*_r(K)$, and thus the action $\alpha$ on $C^*_r(K)$ has the Rokhlin property.
	\end{enumerate}
\end{Exl}

We recall the following characterization of the existence of a unital homomorphism from a dimension drop algebra from \cite[Proposition 5.1]{rordam-winter} (see also \cite[Theorem 4.8]{gardella-hirshberg}).

\begin{Prop} 
	\label{Prop:dim-drop-embedding}
	Let $A$ be a unital $C^*$-algebra. The following are equivalent.
	\begin{enumerate}
		\item The dimension drop algebra $I_{n,n+1}$  admits a unital homomorphism into $A$.
		\item There exists $\eps>0$ and positive mutually orthogonal and mutually Cuntz-equivalent elements $b_1,b_2,\ldots,b_n \in A$ such that $1-\sum_{k=1}^n b_k \precsim (b_1 - \eps)_+$ (where $\precsim$ denotes Cuntz subequivalence).
	\end{enumerate}
\end{Prop}

\begin{Cor}
	Let $A$ be a separable $C^*$-algebra. Then the following are equivalent.
	\begin{enumerate}
		\item $A \cong A \otimes \Zh$.
		\item There exists $n \geq 2$, $\eps>0$ and positive mutually orthogonal and mutually Cuntz-equivalent elements $b_1,b_2,\ldots,b_n \in F(A)$ such that $1-\sum_{k=1}^n b_k \precsim (b_1 - \eps)_+$.
	\end{enumerate}
\end{Cor}

\begin{Thm}
	\label{Thm:NAI-Z-stable}
	Let $A$ be a separable unital $C^*$-algebra. Suppose that $\alpha \in \aut(A)$ has the Rokhlin property and is nearly approximately inner.
		Then $A \rtimes_{\alpha} \Z$ is $\Zh$-stable.
\end{Thm}

\begin{proof} 
	To lighten notation, we write $a \approx_{\eps} b$ to mean $\|a-b\|<\eps$.
	We denote by $u$ the canonical unitary in $A \rtimes_{\alpha} \Z$.
	
	Fix $\eps>0$ and $F \subset A_0$ a finite subset. To simplify some estimates, we assume $\eps<1/14$. We first note that it suffices to exhibit positive orthogonal contractions $b_1,b_2 \in (A \rtimes_{\alpha} \Z)_{\omega}$ such that the following hold:
	\begin{enumerate}
		\item There exists $c \in (A \rtimes_{\alpha}\Z)_{\omega}$ such that $cc^* = b_1$, $c^*c = b_2$ and $\|[c,a]\|<\eps$ for all $a \in F \cup \{u,u^*\}$.
		\item There exists $d \in (A \rtimes_{\alpha} \Z)_{\omega}$ such that $dd^* = 1- (b_1+b_2)$, $d^*d \leq 2(b_1-\frac{1}{2})_+$
		and $\|[d,a]\|<\eps$ for all $d \in F \cup \{u,u^*\}$.
	\end{enumerate}
	If such elements can be found for any $\eps>0$, then by the Kirchberg $\eps$-test we can find elements $b_1,b_2,c,d \in (A \rtimes_{\alpha} \Z)_{\omega} \cap (A \rtimes_{\alpha} \Z)'$ such that $b_1$ and $b_2$ are orthogonal contractions, $cc^* = b_1$, $c^*c = b_2$, $dd^* = 1- (b_1+b_2)$ and $d^*d \leq 2(b_1-\frac{1}{2})_+$. Therefore, by Proposition \ref{Prop:dim-drop-embedding} there exists a unital homomorphism from $I_{2,3}$ to $(A \rtimes_{\alpha} \Z)_{\omega} \cap (A \rtimes_{\alpha} \Z)'$, and therefore, by Lemma \ref{Lemma:DT09}, the crossed product $A \rtimes_{\alpha} \Z$ is $\Zh$-stable.	
	(We remark that the term $2(b_1-\frac{1}{2})_+$ could be replaced by $\frac{1}{t}(b_1 - t)_+$ for any other $t \in (0,1)$; the choice of $t = 1/2$ is arbitrary and makes no difference here.)
	
	Pick $l$ such that $n_l> 1/\eps^2$. Note that $n_l > 14/\eps$.   
	
	Pick a unitary $v \in A$ such that $\|vav^* - \alpha^{n_l}(a)\|<\eps/14$ for all $a \in \bigcup_{k=-7 n_l}^{7 n_l} \alpha^k(F)$. Notice that if $k,j$ satisfy $|j-k| \leq n_l$ 
	then for any $a \in F$, we have
	$$
	\alpha^{k}(v) \alpha^{j}(a) \alpha^{k}(v^*) = \alpha^k(v \alpha^{j-k}(a) v^*) \approx_{\eps/14} \alpha^k(\alpha^{n_l}(\alpha^{j-k}(a)) = \alpha^{n_l+j}(a)
	.
	$$
	
	Pick a double Rokhlin tower $e_{0,0},e_{0,1},\ldots,e_{0,14 n_l-1} ; e_{1,0},e_{1,1},\ldots,e_{1,14 n_l}\in A_{\omega} \cap A'$. For $k=0,1,\ldots, 14 n_l - 1$, we write $f_k = e_{0,k} + e_{1,k}$.
	
	We now proceed to define elements $b_1,b_2,x_1,x_2,c,d \in A_{\omega}\rtimes \Z$. As the definition is somewhat technical, we provide the following illustration to help visualize the construction (which is more accurate in the single tower Rokhlin picture). The $x$-axis should be thought of as indexing the Rokhlin elements, which should be viewed of as roughly $14n_l$-periodic; the arrows are meant to indicate that the illustrated elements implement Cuntz equivalences.

	\

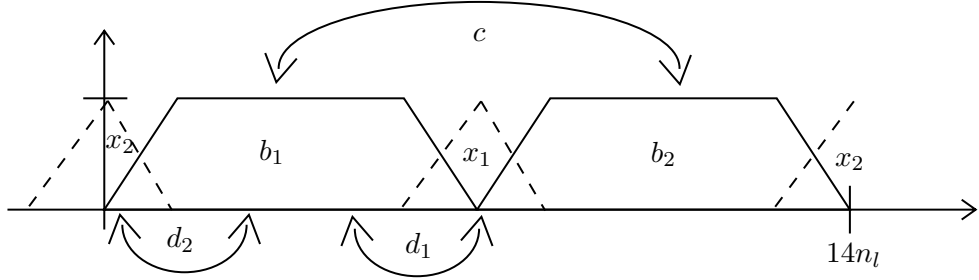
	\tikzset{every picture/.style={line width=0.75pt}} 
	
	\begin{tikzpicture}[x=0.75pt,y=0.75pt,yscale=-1,xscale=1]
	
	\draw    (102.5,146) -- (124.5,146) ;
	\draw   (113,202) -- (149.5,146) -- (262.5,146) -- (299,202) -- cycle ;
	\draw   (299,202) -- (335.5,146) -- (448.5,146) -- (485,202) -- cycle ;
	\draw  (64.69,202) -- (547.79,202)(113,112) -- (113,212) (540.79,197) -- (547.79,202) -- (540.79,207) (108,119) -- (113,112) -- (118,119)  ;
	\draw    (485,190) -- (485,214) ;
	\draw  [dash pattern={on 4.5pt off 4.5pt}]  (301,147.5) -- (259.75,203.25) ;
	\draw  [dash pattern={on 4.5pt off 4.5pt}]  (301,147.5) -- (333,202.5) ;
	\draw  [dash pattern={on 4.5pt off 4.5pt}]  (114.75,147.25) -- (146.75,202.25) ;
	\draw  [dash pattern={on 4.5pt off 4.5pt}]  (487,147.5) -- (445.75,203.25) ;
	\draw  [draw opacity=0] (200.01,131.2) .. controls (200,131.05) and (200,130.9) .. (200,130.75) .. controls (200,112.39) and (244.77,97.5) .. (300,97.5) .. controls (355.23,97.5) and (400,112.39) .. (400,130.75) .. controls (400,130.9) and (400,131.05) .. (399.99,131.2) -- (300,130.75) -- cycle ; \draw   (200.01,131.2) .. controls (200,131.05) and (200,130.9) .. (200,130.75) .. controls (200,112.39) and (244.77,97.5) .. (300,97.5) .. controls (355.23,97.5) and (400,112.39) .. (400,130.75) .. controls (400,130.9) and (400,131.05) .. (399.99,131.2) ;
	\draw   (209.05,125.99) -- (198.83,137.9) -- (193.29,123.21) ;
	\draw  [draw opacity=0] (238,211.61) .. controls (238,211.66) and (238,211.7) .. (238,211.75) .. controls (238,224.87) and (251.88,235.5) .. (269,235.5) .. controls (286.12,235.5) and (300,224.87) .. (300,211.75) .. controls (300,211.7) and (300,211.66) .. (300,211.61) -- (269,211.75) -- cycle ; \draw   (238,211.61) .. controls (238,211.66) and (238,211.7) .. (238,211.75) .. controls (238,224.87) and (251.88,235.5) .. (269,235.5) .. controls (286.12,235.5) and (300,224.87) .. (300,211.75) .. controls (300,211.7) and (300,211.66) .. (300,211.61) ;
	\draw   (231.29,220.29) -- (236.83,205.6) -- (247.05,217.51) ;
	\draw  [draw opacity=0] (122,209.61) .. controls (122,209.66) and (122,209.7) .. (122,209.75) .. controls (122,222.87) and (135.88,233.5) .. (153,233.5) .. controls (170.12,233.5) and (184,222.87) .. (184,209.75) .. controls (184,209.7) and (184,209.66) .. (184,209.61) -- (153,209.75) -- cycle ; \draw   (122,209.61) .. controls (122,209.66) and (122,209.7) .. (122,209.75) .. controls (122,222.87) and (135.88,233.5) .. (153,233.5) .. controls (170.12,233.5) and (184,222.87) .. (184,209.75) .. controls (184,209.7) and (184,209.66) .. (184,209.61) ;
	\draw  [dash pattern={on 4.5pt off 4.5pt}]  (114.75,147.25) -- (73.5,203) ;
	\draw   (174.95,216.51) -- (185.17,204.6) -- (190.71,219.29) ;
	\draw   (405.71,124.21) -- (400.17,138.9) -- (389.95,126.99) ;
	\draw   (290.95,217.51) -- (301.17,205.6) -- (306.71,220.29) ;
	\draw   (115.29,219.29) -- (120.83,204.6) -- (131.05,216.51) ;
	
	\draw (199,173) node   [align=left] {
		$\displaystyle b_{1}$
	};
	\draw (392,174) node   [align=left] {$\displaystyle b_{2}$};
	\draw (299,175.5) node   [align=left] {$\displaystyle x_{1}$};
	\draw (485,178) node   [align=left] {$\displaystyle x_{2}$};
	\draw (487,225) node    {$14n_{l}$};
	\draw (300,114) node    {$c$};
	\draw (270,220) node    {$d_{1}$};
	\draw (121,169) node   [align=left] {$\displaystyle x_{2}$};
	\draw (151,217) node    {$d_{2}$};

	\end{tikzpicture}

	\
	
		Define $h:\{0,1,\ldots,7n_l-1\} \to [0,1]$ by 
	$$
	h(k) = \left \{ 
		\begin{matrix} 
		k/n_l & \mid & k < n_l \\
		1 & \mid & n_l \leq k \leq 6 n_l \\
		1 - (k/n_l - 6 ) & \mid & 6 n_1 + 1 \leq k \leq 7 n_l - 1
		\end{matrix}
	\right .
	$$ 
	
	Now, for $j=1,2$, define 
	$$ 
	b_1 = \sum_{k=0}^{7n_l-1}h(k) f_k  \; , \; 
	b_2 = \sum_{k=0}^{7n_l-1}h(k) f_{k+ 7n_l}
	$$
	
	Notice that $\|[b_j,u]\|< \frac{1}{n_l} < \eps/14$ for $j=1,2$.

	Set 
	$$
	c = \sum_{k=0}^{7n_l-1}h^{1/2}(k)f_k u^{7n_l *}\alpha^{k}(v^{7})
	$$
	We check that
	\begin{align*}
	c^*c & =  \sum_{k,j = 0}^{7n_l-1} h^{1/2}(k) h^{1/2}(j) \alpha^k(v^{7})^* u^{7n_l} f_k f_j u^{7n_l *} \alpha^j(v^{7}) \\
	&  = \sum_{k = 0}^{7n_l-1} h(k) \alpha^k(v^{7})^* u^{7n_l} f_k u^{7n_l *} \alpha^k(v^{7}) \\
	& = \sum_{k = 0}^{7n_l-1}h(k) \alpha^k(v^{7})^* f_{k+ 7n_l} \alpha^k(v^{7}) \\
	& = \sum_{k = 0}^{7n_l-1}h(k) f_{k+7n_l} = b_2
	\end{align*}
	Likewise,
	\begin{align*}
	cc^* & =  \sum_{k,j = 0}^{7n_l-1}h^{1/2}(k)h^{1/2}(j)
	f_ju^{7n_l *}\alpha^j(v^{7})
	\alpha^k(v^{7})^*u^{7n_l}f_k
	\\
	& =
	\sum_{k,j = 0}^{7n_l-1}h^{1/2}(k)h^{1/2}(j)
	u^{7n_l *}\alpha^j(v^{7})
	\alpha^k(v^{7})^*u^{7n_l}f_j f_k \\
	& =   \sum_{k = 0}^{7n_l-1}h(k)
	u^{7n_l *}\alpha^k(v^{7})
	\alpha^k(v^{7})^*u^{7n_l}f_k =
	\sum_{k = 0}^{7n_l-1}h(k)
	f_k  = b_1
	\end{align*}

	Now, let $a \in F$. Then
	\begin{align*}
	ac & = a \sum_{k=0}^{7n_l-1}h^{1/2}(k)f_ku^{7n_l*}\alpha^{k}(v^{7}) \\
	& = \sum_{k=0}^{7n_l-1}h^{1/2}(k)f_k a u^{7n_l*}\alpha^{k}(v^{7}) \\
	& = \sum_{k=0}^{7n_l-1}h^{1/2}(k)f_k u^{7n_l*} \alpha^{7n_l}(a) \alpha^{k}(v^{7}) \\
	& \approx_{\eps/2}  \sum_{k=0}^{7n_l-1}h^{1/2}(k)f_k u^{7n_l*} \alpha^{k} ( v^{7} ) a = ca
	\end{align*}
	So $\|[c,a]\| < \eps/2$. 
	
	As for $a = u$, 
	\begin{align*}
	uc &= u \sum_{k=0}^{7n_l-1}h^{1/2}(k)f_ku^{7n_l*}\alpha^{k}(v^{7}) \\
	& = \sum_{k=0}^{7n_l-1}h^{1/2}(k)f_{k+1} u^{7n_l*}\alpha^{k+1}(v^{7})u
	\end{align*}
	So
	\begin{align*}
	\|uc-cu\| & = \left \|\sum_{k=1}^{7n_l-1} (h^{1/2}(k) - h^{1/2}(k-1)) f_k \right \|  \\
	& = \max_{k=1,2,\ldots,7n_l-1}| h^{1/2}(k) - h^{1/2}(k-1)| = \sqrt{ \frac{1}{n_l}} < \eps
	\end{align*}
	The estimate for $\|[u^*,c]\|$ is similar.

	Define $g_1 \colon \{ 6 n_l , 6 n_l + 1 , \ldots , 8 n_l \} \to [0,1]$ by 
	$$
	g_1(k) = 1 - |7 - k/n_l|
	$$
	Define
	$g_2 \colon \{ 0 , 1 , \ldots , n_l \} \cup \{ 13 n_l , 13 n_l + 1 , \ldots , 14 n_l  \} \to [0,1]$ by 
	$$
	g_2 = \left \{  \
		\begin{matrix}
		1- k/n_l & \mid & k \leq n_l -1 \\
		k/n_l - 13 & \mid & k \geq 13 n_l
		\end{matrix}
	\right .
	$$
	Set 
	$$
	x_1 = \sum_{k=6 n_l}^{8 n_l} g_1(k) f_k
	$$
	and 
	$$
	x_2 = e_{1,14 n_l} +  \sum_{k=0}^{n_l-1} \left ( g_2(k) f_k + g_2(k+ 13 n_l) f_{k+13 n_l} \right )
	$$	
	Notice that $x_1+x_2 = 1-b_1-b_2$. 
	
	Set 
	$$
	d_1 = \sum_{k=0}^{2 n_l - 1} g_1(k+6n_l)^{1/2} f_{k+6n_l} u^{2n_l} \alpha^k(v^{-2})
	$$
	and
	\begin{align*}
		d_2 & =    \sum_{k=0}^{n_l-1} g_2(k)^{1/2} e_{0,k} u^{-2n_l}\alpha^k(v^2) + \sum_{k=0}^{n_l-1}  g_2(k+ 13 n_l)^{1/2} e_{0,k+13 n_l}  u^{-2n_l}\alpha^{k-n_l}(v^2) + \\
		& \quad \quad \quad +
		\sum_{k=0}^{n_l-1} g_2(k)^{1/2} e_{1,k} u^{-2n_l}\alpha^k(v^2) + \sum_{k=0}^{n_l}  g_2(k+ 13 n_l)^{1/2} e_{1,k+13 n_l}  u^{-2n_l}\alpha^{k-n_l-1}(v^2) 
	\end{align*}

	Set $d = d_1 + d_2$. Notice that $d_1d_2 = d_1^*d_2 = d_1d_2^*$, and therefore $d^*d = d_1^*d_1 + d_2^*d_2$, and $dd^* = d_1d_1^* + d_2 d_2^*$. 
	
	We have 
		\begin{align*}
	d_1^*d_1 
	&= \sum_{k=0}^{2 n_l - 1}  \sum_{j=0}^{2 n_l - 1} g_1(k+6n_l)^{1/2}g_1(j+6n_l)^{1/2} \alpha^j(v^{-2})^*u^{-2n_l}  f_{j+6n_l}   f_{k+6n_l} u^{2n_l} \alpha^k(v^{-2})   \\
	&= \sum_{k=0}^{2 n_l - 1}  g_1(k+6n_l) \alpha^k(v^{-2})^*u^{-2n_l}   f_{k+6n_l} u^{2n_l} \alpha^k(v^{-2})   \\
	&= \sum_{k=0}^{2 n_l - 1}  g_1(k+6n_l) \alpha^k(v^{-2})^*   f_{k+4n_l}\alpha^k(v^{-2})   \\
	&= \sum_{k=0}^{2 n_l - 1}  g_1(k+6n_l)   f_{k+4n_l}  \\
	&\leq \sum_{k=4n_l}^{6n_l-1} f_k .
	\end{align*}
	Likewise, we have
	$$
	d_2^*d_2 \leq \sum_{k=n_l}^{3n_l-1} f_k
	$$
	so 
	$$
	d^*d \leq \sum_{k=n_l}^{6n_l-1} f_k \leq 2(b_1 -1/2)_+   .
	$$
	Now, 
	\begin{align*}
	d_1d_1^* 
		&= \sum_{k=0}^{2 n_l - 1}  \sum_{j=0}^{2 n_l - 1} g_1(k+6n_l)^{1/2}g_1(j+6n_l)^{1/2}  f_{k+6n_l} u^{2n_l} \alpha^k(v^{-2}) \alpha^j(v^{-2})^*u^{-2n_l}  f_{j+6n_l}   \\
		&= \sum_{k=0}^{2 n_l - 1}  \sum_{j=0}^{2 n_l - 1} g_1(k+6n_l)^{1/2}g_1(j+6n_l)^{1/2}  f_{k+6n_l} f_{j+6n_l}  \alpha^{k+2n_l}(v^{-2}) \alpha^{j+2n_l}(v^{-2})^* \\
		&= \sum_{k=0}^{2 n_l - 1}  g_1(k+6n_l)  f_{k+6n_l}  \alpha^{k+2n_l}(v^{-2}) \alpha^{k+2n_l}(v^{-2})^* \\
		&= x_1
	\end{align*}
	and likewise, noting that for any two relevant indices $k,j$ we have 
	$$
	u^{-2n_l} \alpha^k(v^2)\alpha^{-j}(v^{-2})  u^{2n_l} = \alpha^{k-2n_l}(v^2)\alpha^{-j-2n_l}(v^{-2}) \in A ,
	$$
	 this element commutes with the Rokhlin projections, and the Rokhlin projections are pairwise orthogonal, we obtain:
	$$
	d_2d_2^* = x_2 .
	$$
	Thus, $dd^* = 1-b_1-b_2$.

	For any $a \in F$, using the fact that $a u^{2n_l} \alpha^k(v^{-2}) \approx_{\eps/7} u^{2n_l} \alpha^k(v^{-2})a$,
	we have
	\begin{align*}
	a d_1 
	&= \sum_{k=0}^{2 n_l - 1} g_1(k+6n_l)^{1/2} f_{k+6n_l} a u^{2n_l} \alpha^k(v^{-2}) \\
	&\approx_{\eps/7} \sum_{k=0}^{2 n_l - 1} g_1(k+6n_l)^{1/2} f_{k+6n_l} u^{2n_l} \alpha^k(v^{-2})a  \\
	&= d_1a 
	\end{align*}
	and likewise one sees that $\|ad_2 - d_2a\| < \eps/7$ for all $a \in F$. 

	We now estimate $\|[u,d_1]\|$. Note that 
	\begin{align*}
	u d_1 
	&= \sum_{k=0}^{2 n_l - 1} g_1(k+6n_l)^{1/2} u f_{k+6n_l} u^{2n_l} \alpha^k(v^{-2}) \\
	&= \sum_{k=0}^{2 n_l - 1} g_1(k+6n_l)^{1/2} f_{k+6n_l+1} u^{2n_l+1}\alpha^k(v^{-2})
	\end{align*}
	and
		\begin{align*}
	 d_1 u
	&= \sum_{k=0}^{2 n_l - 1} g_1(k+6n_l)^{1/2} f_{k+6n_l} u^{2n_l} \alpha^k(v^{-2}) u \\
	&= \sum_{k=0}^{2 n_l - 1} g_1(k+6n_l)^{1/2} f_{k+6n_l} u^{2n_l+1}\alpha^{k-1}(v^{-2}) \\
	\end{align*}
	Thus, 
	$$
	\|ud_1-d_1u\| 
	 = \max_{k=6n_l,6n_1+1,\ldots,8n_l}| g_1^{1/2}(k) - g_1^{1/2}(k-1)| = \sqrt{ \frac{1}{n_l}} < \eps .
	$$
	In a similar way, we see that $\|u^*d_1-d_1u^*\| < \eps$ as well.
	
	It remains to estimate $\|[u,d_2]\|$, which is somewhat more delicate.  We have
	\begin{align*}
	u d_2 & =  & & \sum_{k=0}^{n_l-1} g_2(k)^{1/2} u e_{0,k} u^{-2n_l}\alpha^k(v^2) \\
	& &+& \sum_{k=0}^{n_l-1}  g_2(k+ 13 n_l)^{1/2} u e_{0,k+13 n_l}  u^{-2n_l}\alpha^{k-n_l}(v^2) \\
	& &+&
	\sum_{k=0}^{n_l-1} g_2(k)^{1/2} u e_{1,k} u^{-2n_l}\alpha^k(v^2) \\
	& &+& \sum_{k=0}^{n_l}  g_2(k+ 13 n_l)^{1/2} u e_{1,k+13 n_l}  u^{-2n_l}\alpha^{k-n_l-1}(v^2) \\
	&= & &\sum_{k=0}^{n_l-1} g_2(k)^{1/2} e_{0,k+1} u^{-2n_l+1}\alpha^k(v^2) \\
	& &+&
	\sum_{k=0}^{n_l-2}  g_2(k+ 13 n_l)^{1/2}  e_{0,k+13 n_l+1}  u^{-2n_l+1}\alpha^{k-n_l}(v^2)  \\
	& &+&
	\sum_{k=0}^{n_l-1} g_2(k)^{1/2} e_{1,k+1} u^{-2n_l+1}\alpha^k(v^2) \\
	& &+& \sum_{k=0}^{n_l-1}  g_2(k+ 13 n_l)^{1/2} e_{1,k+13 n_l+1}  u^{-2n_l+1}\alpha^{k-n_l-1}(v^2)  \\
	& &+& (1-1/n_l)^{1/2} u e_{0,14 n_1 - 1} u^{-2n_l}\alpha^{-1}(v^2) 
		+ u e_{1,14 n_1 } u^{-2n_l}\alpha^{-1}(v^2) 
	\end{align*}
	while
	\begin{align*}
	d_2 u& =  & & \sum_{k=0}^{n_l-1} g_2(k)^{1/2}  e_{0,k} u^{-2n_l}\alpha^k(v^2) u \\
	& &+& \sum_{k=0}^{n_l-1}  g_2(k+ 13 n_l)^{1/2}  e_{0,k+13 n_l}  u^{-2n_l}\alpha^{k-n_l}(v^2) u \\
	& &+&
	\sum_{k=0}^{n_l-1} g_2(k)^{1/2}  e_{1,k} u^{-2n_l}\alpha^k(v^2)u \\
	& &+& \sum_{k=0}^{n_l}  g_2(k+ 13 n_l)^{1/2}  e_{1,k+13 n_l}  u^{-2n_l}\alpha^{k-n_l-1}(v^2) u\\
	&= & &\sum_{k=1}^{n_l-1} g_2(k)^{1/2} e_{0,k} u^{-2n_l+1}\alpha^{k-1}(v^2) \\
	& &+&
	\sum_{k=0}^{n_l-1}  g_2(k+ 13 n_l)^{1/2}  e_{0,k+13 n_l}  u^{-2n_l+1}\alpha^{k-n_l-1}(v^2)  \\
	& &+&
	\sum_{k=1}^{n_l-1} g_2(k)^{1/2} e_{1,k} u^{-2n_l+1}\alpha^{k-1}(v^2) \\
	& &+& \sum_{k=0}^{n_l}  g_2(k+ 13 n_l)^{1/2} e_{1,k+13 n_l}  u^{-2n_l+1}\alpha^{k-n_l-1}(v^2)  \\
	& &+& e_{0,0} u^{-2n_l+1}\alpha^{-1}(v^2) 
	+  e_{1,0 } u^{-2n_l+1}\alpha^{-1}(v^2) 
	\end{align*}
	Now, 
	\begin{align*}
	\left \| ( (1-1/n_l)^{1/2} u e_{0,14 n_1 - 1} 
	+ u e_{1,14 n_1 }  ) - 
		(e_{0,0} u
		+  e_{1,0 } u ) \right  \| &= 1-(1-1/n_l)^{1/2} \\
		&< \frac{1}{n_l} < \eps
	\end{align*}
	The rest of the terms in the expression $ud_2 - d_2u$ can be compared term by term as is done in the case the corresponding estimate for $d_1$, and each is bounded above by $\sqrt{1/n_l}$. It follows that $\|ud_2 - d_2 u \|< \eps$ as well. In a similar manner, one checks that $\|ud_2 - d_2 u \|< \eps$. This completes the proof.
\end{proof}
 \begin{Question}
 	Does the conclusion of Theorem \ref{Thm:NAI-Z-stable} hold if we replace the Rokhlin property by finite Rokhlin dimension, with or without commuting towers?
 \end{Question}


\begin{thebibliography}{10}
	
	\bibitem{AGP2019}
	M.~Ali Asadi-Vasfi, Nasser Golestani, and N.~Christopher Phillips.
	\newblock The {C}untz semigroup and the radius of comparison of the crossed
	product by a finite group.
	\newblock Preprint, arXiv:1908.06343,, 2019.
	
	\bibitem{CETWW}
	Jorge Castillejos, Samuel Evington, Aaron Tikuisis, Stuart White, and Wilhelm
	Winter.
	\newblock Nuclear dimension of simple {$C^*$}-algebras.
	\newblock Preprint, arXiv:1901.05853, 2019.
	
	\bibitem{DHTW}
	Marius Dadarlat, Ilan Hirshberg, Andrew~S. Toms, and Wilhelm Winter.
	\newblock The {J}iang-{S}u algebra does not always embed.
	\newblock {\em Math. Res. Lett.}, 16(1):23--26, 2009.
	
	\bibitem{Dadarlat-Toms}
	Marius Dadarlat and Andrew~S. Toms.
	\newblock {$\mathcal{Z}$}-stability and infinite tensor powers of
	{$C^*$}-algebras.
	\newblock {\em Adv. Math.}, 220(2):341--366, 2009.
	
	\bibitem{Dixmier-Lance}
	Jacques Dixmier and E.~Christopher Lance.
	\newblock Deux nouveaux facteurs de type {${\rm II}_{1}$}.
	\newblock {\em Invent. Math.}, 7:226--234, 1969.
	
	\bibitem{EGLN:arXiv}
	George~A. Elliott, Guihua Gong, Huaxin Lin, and Zhuang Niu.
	\newblock On the classification of simple amenable {$C^*$}-algebras with finite
	decomposition rank, {II}.
	\newblock Preprint, arXiv:1507.03437v2, 2015.
	
	\bibitem{Elliott-Niu-mean-dimension}
	George~A. Elliott and Zhuang Niu.
	\newblock The {$C^*$}-algebra of a minimal homeomorphism of zero mean
	dimension.
	\newblock {\em Duke Math. J.}, 166(18):3569--3594, 2017.
	
	\bibitem{gardella-hirshberg}
	Eusebio Gardella and Ilan Hirshberg.
	\newblock Strongly outer actions of amenable groups on {$\mathcal{Z}$}-stable
	{$C^*$}-algebras.
	\newblock Preprint, arXiv:1811.00447, 2018.
	
	\bibitem{GHS}
	Eusebio Gardella, Ilan Hirshberg, and Luis Santiago.
	\newblock Rokhlin dimension: duality, tracial properties, and crossed products.
	\newblock Preprint, arXiv:1709.00222, to appear, \emph{Ergodic Theory Dynam.
		Systems}, 2017.
	
	\bibitem{giol-kerr}
	Julien Giol and David Kerr.
	\newblock Subshifts and perforation.
	\newblock {\em J. Reine Angew. Math.}, 639:107--119, 2010.
	
	\bibitem{Hewitt-Ross}
	Edwin Hewitt and Kenneth~A. Ross.
	\newblock {\em Abstract harmonic analysis. {V}ol. {I}}, volume 115 of {\em
		Grundlehren der Mathematischen Wissenschaften [Fundamental Principles of
		Mathematical Sciences]}.
	\newblock Springer-Verlag, Berlin-New York, second edition, 1979.
	\newblock Structure of topological groups, integration theory, group
	representations.
	
	\bibitem{hirshberg-orovitz}
	Ilan Hirshberg and Joav Orovitz.
	\newblock Tracially {$\Zh$}-absorbing {$C^*$}-algebras.
	\newblock {\em J. Funct. Anal.}, 265(5):765--785, 2013.
	
	\bibitem{hirshberg-phillips}
	Ilan Hirshberg and N.~Christopher Phillips.
	\newblock {R}okhlin dimension: obstructions and permanence properties.
	\newblock {\em Doc. Math.}, 20:199--236, 2015.
	
	\bibitem{hirshberg-phillips-asymmetry}
	Ilan Hirshberg and N.~Christopher Phillips.
	\newblock A simple nuclear {$C^*$}-algebra with an internal asymmetry.
	\newblock Preprint, arXiv:1909.10728, 2019.
	
	\bibitem{HRW}
	Ilan Hirshberg, Mikael R{\o}rdam, and Wilhelm Winter.
	\newblock {$C_0(X)$}-algebras, stability and strongly self-absorbing
	{$C^*$}-algebras.
	\newblock {\em Math. Ann.}, 339(3):695--732, 2007.
	
	\bibitem{HSWW}
	Ilan Hirshberg, G\'{a}bor Szab\'{o}, Wilhelm Winter, and Jianchao Wu.
	\newblock Rokhlin dimension for flows.
	\newblock {\em Comm. Math. Phys.}, 353(1):253--316, 2017.
	
	\bibitem{HW-PJM}
	Ilan Hirshberg and Wilhelm Winter.
	\newblock Rokhlin actions and self-absorbing {$C^*$}-algebras.
	\newblock {\em Pacific J. Math.}, 233:125--143, 2007.
	
	\bibitem{HWZ}
	Ilan Hirshberg, Wilhelm Winter, and Joachim Zachrias.
	\newblock Rokhlin dimension and {$C^*$}-dynamics.
	\newblock {\em Comm. Math. Phys.}, 335:637--670, 2015.
	
	\bibitem{Hirshberg-Wu}
	Ilan Hirshberg and Jianchao Wu.
	\newblock The nuclear dimension of {$C^*$}-algebras associated to
	homeomorphisms.
	\newblock {\em Adv. Math.}, 304:56--89, 2017.
	
	\bibitem{Jiang-Su}
	Xinhui Jiang and Hongbing Su.
	\newblock On a simple unital projectionless {$C^*$}-algebra.
	\newblock {\em Amer. J. Math.}, 121(2):359--413, 1999.
	
	\bibitem{kerr-szabo}
	David Kerr and G{\'a}bor Szab{\'o}.
	\newblock Almost finiteness and the small boundary property.
	\newblock Preprint, arXiv:1807.04326, 2018.
	
	\bibitem{Kirchberg-abel}
	Eberhard Kirchberg.
	\newblock Central sequences in {$C^*$}-algebras and strongly purely infinite
	algebras.
	\newblock In {\em Operator {A}lgebras: {T}he {A}bel {S}ymposium 2004}, volume~1
	of {\em Abel Symp.}, pages 175--231. Springer, Berlin, 2006.
	
	\bibitem{kirchberg-rordam-non-simple}
	Eberhard Kirchberg and Mikael R{\o}rdam.
	\newblock Non-simple purely infinite {$C^*$}-algebras.
	\newblock {\em Amer. J. Math.}, 122(3):637--666, 2000.
	
	\bibitem{kishimoto-rokhlin}
	Akitaka Kishimoto.
	\newblock The {R}ohlin property for automorphisms of {UHF} algebras.
	\newblock {\em J. Reine Angew. Math.}, 465:183--196, 1995.
	
	\bibitem{loring}
	Terry~A. Loring.
	\newblock {\em Lifting solutions to perturbing problems in {$C^*$}-algebras},
	volume~8 of {\em Fields Institute Monographs}.
	\newblock American Mathematical Society, Providence, RI, 1997.
	
	\bibitem{Matui-Sato-Z-stability}
	Hiroki Matui and Yasuhiko Sato.
	\newblock {$\mathcal{Z}$}-stability of crossed products by strongly outer
	actions.
	\newblock {\em Comm. Math. Phys.}, 314(1):193--228, 2012.
	
	\bibitem{phillips-tracial}
	N.~Christopher {Phillips}.
	\newblock {The tracial Rokhlin property for actions of finite groups on
		{$C^*$}-algebras.}
	\newblock {\em {Amer. J. Math.}}, 133(3):581--636, 2011.
	
	\bibitem{Robert-Rordam}
	Leonel Robert and Mikael R{\o}rdam.
	\newblock Divisibility properties for {$C^*$}-algebras.
	\newblock {\em Proc. Lond. Math. Soc. (3)}, 106(6):1330--1370, 2013.
	
	\bibitem{Rordam-stable-real-rank}
	Mikael R{\o}rdam.
	\newblock The stable and the real rank of {$\mathcal{Z}$}-absorbing
	{$C^*$}-algebras.
	\newblock {\em Internat. J. Math.}, 15(10):1065--1084, 2004.
	
	\bibitem{rordam-winter}
	Mikael R{\o}rdam and Wilhelm Winter.
	\newblock The {J}iang-{S}u algebra revisited.
	\newblock {\em J. Reine Angew. Math.}, 642:129--155, 2010.
	
	\bibitem{Sato}
	Yasuhiko Sato.
	\newblock The {R}ohlin property for automorphisms of the {J}iang-{S}u algebra.
	\newblock {\em J. Funct. Anal.}, 259(2):453--476, 2010.
	
	\bibitem{Szabo}
	G{\'a}bor Szab{\'o}.
	\newblock The {R}okhlin dimension of topological {$\mathbb{Z}^m$}-actions.
	\newblock {\em Proc. Lond. Math. Soc.}, 110(3):673--694, 2015.
	
	\bibitem{SWZ}
	G\'{a}bor Szab\'{o}, Jianchao Wu, and Joachim Zacharias.
	\newblock Rokhlin dimension for actions of residually finite groups.
	\newblock {\em Ergodic Theory Dynam. Systems}, 39(8):2248--2304, 2019.
	
	\bibitem{tikuisis-non-unital}
	Aaron Tikuisis.
	\newblock Nuclear dimension, {$\mathcal{Z}$}-stability, and algebraic
	simplicity for stably projectionless {$C^*$}-algebras.
	\newblock {\em Math. Ann.}, 358(3-4):729--778, 2014.
	
	\bibitem{TWW}
	Aaron Tikuisis, Stuart White, and Wilhelm Winter.
	\newblock Quasidiagonality of nuclear {$C^\ast$}-algebras.
	\newblock {\em Ann. of Math. (2)}, 185(1):229--284, 2017.
	
	\bibitem{toms-counterexample}
	Andrew~S. Toms.
	\newblock On the classification problem for nuclear {$C^\ast$}-algebras.
	\newblock {\em Ann. of Math. (2)}, 167(3):1029--1044, 2008.
	
	\bibitem{Toms-Winter-GAFA}
	Andrew~S. Toms and Wilhelm Winter.
	\newblock Minimal dynamics and {K}-theoretic rigidity: {E}lliott's conjecture.
	\newblock {\em Geom. Funct. Anal.}, 23(1):467--481, 2013.
	
	\bibitem{Willimas-1981}
	Dana~P. Williams.
	\newblock Transformation group {$C^{\ast} $}-algebras with continuous trace.
	\newblock {\em J. Funct. Anal.}, 41(1):40--76, 1981.
	
	\bibitem{Winter-nuc-dim-implies-Z-stable}
	Wilhelm Winter.
	\newblock Nuclear dimension and {$\mathcal{Z}$}-stability of pure
	{$C^*$}-algebras.
	\newblock {\em Invent. Math.}, 187(2):259--342, 2012.
	
	\bibitem{winter-zacharias-order-zero}
	Wilhelm Winter and Joachim Zacharias.
	\newblock Completely positive maps of order zero.
	\newblock {\em M\"{u}nster J. Math.}, 2:311--324, 2009.
	
\end{thebibliography}

\end{document}